\RequirePackage{luatex85}
\documentclass[a4paper,oneside,reqno]{amsart}

\usepackage[english]{babel}
\usepackage[utf8]{inputenc}		
\usepackage{lmodern}
\usepackage[scaled]{helvet}		
\usepackage{courier}			
\usepackage{eulervm}			
\usepackage[bb=boondox,bbscaled=1.05,scr=dutchcal]{mathalfa}	%
\usepackage{dsfont}
\normalfont

\usepackage{indentfirst}				
\usepackage{tabularx,longtable,booktabs,multirow}			
\usepackage{caption,subcaption}			
\usepackage{fullpage}				    
\usepackage[english]{varioref}			
\usepackage[dvipsnames]{xcolor}			
\usepackage[hyperindex,breaklinks]{hyperref}		
\usepackage{doi}
\usepackage{bookmark}					
\usepackage{textcomp}
\usepackage{rotating,afterpage,pdflscape}
\usepackage{pbox}
\usepackage[all]{xy}

\usepackage{graphicx}		
\usepackage{float}
\usepackage{tikz}			
\usepackage{tikz-cd}		
\usetikzlibrary{
	arrows,decorations.pathreplacing,decorations.markings,decorations.pathmorphing,shapes.geometric,positioning,calc,fadings,angles
	}
\tikzfading[name=inner fade, inner color=transparent!0, outer color=transparent!100]
\graphicspath{{Images/}}
\usepackage[export]{adjustbox}

\usepackage{amsmath,amssymb,amsthm,mathtools}
\usepackage{bm,braket,faktor,esint,stmaryrd,bbm}
\usepackage{tcolorbox}
\usepackage[export]{adjustbox}

\definecolor{webbrown}{rgb}{0.65, 0.16, 0.16}

\hypersetup{
colorlinks=true, linktocpage=true, pdfstartpage=1, pdfstartview=FitV,breaklinks=true, pdfpagemode=UseNone, pageanchor=true, pdfpagemode=UseOutlines,plainpages=false, bookmarksnumbered, bookmarksopen=true, bookmarksopenlevel=1,hypertexnames=true, pdfhighlight=/O,urlcolor=webbrown, linkcolor=RoyalBlue, citecolor=ForestGreen}

\usepackage{enumerate,enumitem}
\numberwithin{equation}{section}

\makeatletter
\def\l@subsection{\@tocline{2}{0pt}{2.5pc}{5pc}{}}
\makeatother

\usepackage[textsize=footnotesize]{todonotes}
\presetkeys%
		{todonotes}%
		{color=Apricot}{}%
\usepackage[style=alphabetic,maxalphanames=5,giveninits,sorting=nyt,%
maxbibnames=99,backend=biber]{biblatex}
\renewbibmacro{in:}{}
\usepackage{csquotes}
\usepackage[noabbrev]{cleveref}
\expandafter\def\csname ver@etex.sty\endcsname{3000/12/31}

\crefname{lemma}{lemma}{lemmata}
\Crefname{lemma}{Lemma}{Lemmata}
\crefname{introthm}{theorem}{theorems}
\Crefname{introthm}{Theorem}{Theorems}
\crefname{subsection}{subsection}{subsections}
\Crefname{subsection}{Subsection}{Subsections}
\crefname{conjecture}{conjecture}{conjectures}
\Crefname{conjecture}{Conjecture}{Conjectures}
\crefname{question}{question}{questions}
\Crefname{question}{Question}{Questions}
\crefname{warning}{warning}{warnings}
\Crefname{warning}{Warning}{Warnings}

\newcommand{\mf}[1]{\mathfrak{#1}}
\newcommand{\mc}[1]{\mathcal{#1}}
\newcommand{\mb}[1]{\mathbb{#1}}
\newcommand{\ms}[1]{\mathsf{#1}}

\renewcommand{\theta}{\vartheta}
\renewcommand{\epsilon}{\varepsilon}
\renewcommand{\phi}{\varphi}
\DeclareSymbolFont{greekletters}{OML}{cmr}{m}{it}
\DeclareMathSymbol{\varsigma}{\mathalpha}{greekletters}{"26}	
\usepackage{pdftexcmds}						
\DeclareFontFamily{U}{cbgreek}{}
\DeclareFontShape{U}{cbgreek}{m}{n}{
				<-6>    grmn0500
				<6-7>   grmn0600
				<7-8>   grmn0700
				<8-9>   grmn0800
				<9-10>  grmn0900
				<10-12> grmn1000
				<12-17> grmn1200
				<17->   grmn1728
			}{}
\DeclareFontShape{U}{cbgreek}{bx}{n}{
				<-6>    grxn0500
				<6-7>   grxn0600
				<7-8>   grxn0700
				<8-9>   grxn0800
				<9-10>  grxn0900
				<10-12> grxn1000
				<12-17> grxn1200
				<17->   grxn1728
			}{}
\DeclareRobustCommand{\qoppa}{%
	\text{\usefont{U}{cbgreek}{\normalorbold}{n}\symbol{19}}%
}

\makeatletter
\newcommand{\normalorbold}{%
	\ifnum\pdf@strcmp{\math@version}{bold}=\z@ bx\else m\fi
}
\makeatother

\newcommand{\bbraket}[1]{\llbracket #1 \rrbracket}
\newcommand{\pparenthesis}[1]{(\!( #1 )\!)}
\newcommand{\normord}[1]{{\vcentcolon}\!\mathrel{#1}\!{\vcentcolon}}

\newcommand{\<}{\langle}
\renewcommand{\>}{\rangle}

\newcommand{\Z}{\mathbb{Z}}
\newcommand{\Q}{\mathbb{Q}}
\newcommand{\R}{\mathbb{R}}
\newcommand{\C}{\mathbb{C}}

\renewcommand{\P}{\mb{P}}
\newcommand{\M}{\overline{\mc{M}}}
\DeclareMathOperator{\SP}{\mc{SP}}
\DeclareMathOperator{\OP}{\mc{OP}}

\DeclareMathOperator{\Cl}{\mc{C}\ell}

\DeclareMathOperator{\ev}{ev}

\DeclareMathOperator{\sh}{sh}
\newcommand{\Aut}{\mathord{\mathrm{Aut}}}
\newcommand{\Strlng}[3]{\genfrac\{\}{0pt}{}{#1}{#2}_{#3}}		

\newcommand{\vir}{{\rm vir}}
\newcommand{\loc}{{\rm loc}}

\newcommand{\sGWconn}{\mb{G}^{\circ}}
\newcommand{\sGWdisc}{\mb{G}^{\bullet}}
\newcommand{\dHconn}{\mb{H}^{\circ}}
\newcommand{\dHdisc}{\mb{H}^{\bullet}}

\newcommand{\del}{\partial}				

\newcommand{\Id}{{\mathrm{Id}}}

\newcommand{\pt}{{\rm pt}}
\newcommand{\val}{{\rm val}}
\DeclareMathOperator{\cont}{cont}

\theoremstyle{plain}
\newtheorem{theorem}{Theorem}[section]
\newtheorem{proposition}[theorem]{Proposition}
\newtheorem{lemma}[theorem]{Lemma}
\newtheorem{corollary}[theorem]{Corollary}

\theoremstyle{definition}
\newtheorem{definition}[theorem]{Definition}
\theoremstyle{remark}
\newtheorem{remark}[theorem]{Remark}

\theoremstyle{plain}
\newtheorem{introthm}{Theorem}

\newtheorem{introconj}[introthm]{Conjecture}

\title{
	The spin Gromov--Witten/Hurwitz correspondence for $\mathbb{P}^1$
}
\author[A.~Giacchetto]{A.~Giacchetto}
\address[A.~G.]{ 
	Universit\'e Paris-Saclay, CNRS, CEA, Institut de Physique Th\'eorique, 91191 Gif-sur-Yvette, France %
}
\curraddr{ 
	ETH Zürich, Departement Mathematik, 8044 Zürich, Switzerland %
}
\email{alessandro.giacchetto@math.ethz.ch}
\author[R.~Kramer]{R.~Kramer}
\address[R.~K.]{
	University of Alberta, Edmonton, AB, T6G 2G1, Canada
	}
\curraddr{
	Università degli Studi di Milano-Bicocca, Dipartimento di Matematica e Applicazioni, 20125 Milano, Italy
	}
\email{reinier.kramer@unimib.it}
\author[D.~Lewa{\'{n}}ski]{D.~Lewa{\'{n}}ski}
\address[D.~L.]{
	Université de Genève, Section de Mathématiques, rue de Conseil-Général 7, 1206 Genève, Switzerland
}
\curraddr{
	Università degli Studi di Trieste, Dipartimento di Matematica, Informatica e Geoscienze, 34127 Trieste, Italy
}
\email{danilo.lewanski@units.it}
\author[A.~Sauvaget]{A.~Sauvaget}
\address[A.~S.]{
	Laboratoire de Mathématiques AGM (UMR 8088), Université de Cergy-Pontoise, 2 av. Adolphe Chauvin, 95302, Cergy-Pontoise, France
}
\email{adrien.sauvaget@math.cnrs.fr}

\begin{document}

\begin{abstract}
	We study the spin Gromov--Witten (GW) theory of $\mathbb{P}^1$. Using the standard torus action on $\mathbb{P}^1$, we prove that the associated equivariant potential can be expressed by means of operator formalism and satisfies the 2-BKP hierarchy. As a consequence of this result, we prove the spin analogue of the GW/Hurwitz correspondence of Okounkov--Pandharipande for $\mathbb{P}^1$, which was conjectured by J.~Lee. Finally, we prove that this correspondence for a general target spin curve follows from a conjectural degeneration formula for spin GW invariants that holds in virtual dimension 0.
\end{abstract}

\maketitle
\vspace{-.5cm}
\tableofcontents

\section{Introduction}
\label{sec:intro}

\subsection{Gromov--Witten theory of spin curves}

The Gromov--Witten (GW) theory of K\"{a}hler surfaces has seen significant development over the last twenty years. Initially focused on rational and ruled surfaces of geometric genus $p_g = 0$, it later extended to surfaces with positive $p_g$, a class that includes Calabi--Yau surfaces, most elliptic surfaces and most surfaces of general type \cite{IP04, LP07, KL07, KL11a, KL11b, Lee13}.

Let $X$ be such an algebraic surface with a homology class $\beta \in H_2(X,\Z)$, and a holomorphic $2$-form $\eta$. If the zero locus $B$ of $\eta$ is smooth and connected, then $(B,N_{B/X})$ is a {\em spin curve}, i.e. a pair consisting of a curve and a line bundle $\vartheta$ such that $\vartheta^2 \cong \omega_B$ (a theta characteristic). One of the striking ideas that emerged is the reduction of the GW theory of surfaces of positive geometric genus to the GW theory of a spin curve associated to it via the so-called localisation by cosections that we now recall (see \cite{KL07, LP07} for surfaces with smooth canonical divisor, and \cite[conjecture~1.2]{KL11b} for a lift of that assumption).

The 2-form $\eta$ induces a morphism from the obstruction sheaf of the moduli of stable maps to $X$ to its structure sheaf (a cosection) \cite{KL13}:
\begin{equation}
	\sigma_{\eta} \colon
	\mc{O}b_{\M_{g,n}(X, \beta)} \longrightarrow \mc{O}_{\M_{g,n}(X, \beta)} .
\end{equation}
The locus $Z(\sigma_{\eta})$ where $\sigma_{\eta}$ is not surjective is the locus of stable maps to $X$ whose image lies in the degeneracy locus of $\eta$. This locus is empty if $\beta$ is not a multiple of the canonical class. Applying the general theory of localisation by cosections, Kiem and Li showed that the virtual fundamental cycle of $\M_{g,n}(X, \beta)$ can be represented as a cycle supported on $Z(\sigma_{\eta})$. In other words, they constructed a \textit{localised} virtual cycle on the locus $Z(\sigma_{\eta})$ such that
\begin{equation}
	\iota_* \bigl[Z(\sigma_{\eta})\bigr]^{\textup{loc}, \eta} = \bigl[ \M_{g,n}(X,\beta) \bigr]^{\textup{vir}} ,
\end{equation}
where $\iota \colon Z(\sigma_{\eta}) \hookrightarrow \M_{g,n}(X,\beta)$ is the embedding morphism. In particular, the GW invariants of $X$ vanish unless $\beta = d \cdot K_X$. This approach is an algebro-geometric analogue of what was first discovered by Lee and Parker in symplectic geometry \cite{LP07}. 

If the vanishing locus of $\eta$ is smooth and $\beta = d \cdot K_X$, then we have the canonical identification:
\begin{equation}
	Z(\sigma_{\eta}) \simeq \M_{g,n}(B, d) ,
\end{equation}
and we write $\bigl[\M_{g,n}(B, d)\bigr]^{\textup{loc}, \vartheta} = \bigl[Z(\sigma_{\eta})\bigr]^{\textup{loc}, \eta}$, where $\vartheta = N_{B/X}$. If $\M_{g,n}(X, d[B])$ is proper, then the GW theory of $X$ can be reduced to the localised GW theory of $B$ in the following way:
\begin{equation}\label{eqn:reduction}
	\int_{[\M_{g,n}(X,d[B])]^{\vir}} \prod_{i=1}^n \psi_i^{k_i} \ev_i^{\ast}({\gamma_i})
	=
	\int_{[\M_{g,n}(B,d)]^{\loc,\vartheta}} \prod_{i=1}^n \psi_i^{k_i} \ev_i^{\ast}(\gamma_i \cdot [B]) .
\end{equation}
Here, each $k_i$ is a non-negative integer and the $\gamma_i$ are classes in $H^*(X,\Z)$ (see~\cite{KL13}). The cycle $[\M_{g,n}(B,d)]^{\loc,\vartheta}$ only depends on the spin curve, therefore we refer to the invariants in the right-hand side of~\cref{eqn:reduction} as the {\em spin GW invariants} of $(B,\vartheta)$. The deformation invariance proved in~\cite{KL13} implies that these invariants only depend on the genus of $B$ and the {\em parity} (or Arf invariant) of $\vartheta$, i.e. the parity of $h^0(B,\vartheta)$.

The structure of these invariants has been a topic of research since the result of Lee--Parker \cite{LP07}, and first conjectures were made by Maulik--Pandharipande \cite{MP08}. A recurring idea since these pioneering works, is that they should have a similar structure to the classical GW invariants of curves as described by Okounkov--Pandharipande in their trilogy \cite{OP06a,OP06b,OP06c}. The most concrete proposal in this direction is Lee's conjecture, which states that spin ({\em stationary}) GW invariants may be recovered by so-called spin Hurwitz numbers (see~\cite{Lee20} and~\cref{conj:full:spin:GW:H} below). However, several steps in the general program of Okounkov--Pandharipande are significantly more difficult for spin curves, as we will indicate in the rest of this introduction.

\subsection{Equivariant spin Gromov--Witten theory of \texorpdfstring{$\P^1$}{the Riemann sphere}}

In general, one cannot relate the localised and the classical virtual cycles. However, for the special case $(B,\vartheta) = (\P^1,\mc{O}(-1))$, the following relation holds in $A_{g-1+d+n}(\M_{g,n}(\P^1,d))$ (see~\cite[corollary 2.4]{KL11a}):
\begin{equation}\label{eqn:locsplit}
	\bigl[
		\M_{g,n}(\P^1,d)
	\bigr]^{\loc,\mc{O}(-1)}
	=
	c_{g-1+d}\left( R^1\pi_{\ast} f^{\ast}\mc{O}(-1) \right)
	\cdot
	\bigl[
		\M_{g,n}(\P^1,d)
	\bigr]^\vir ,
\end{equation}
where $\pi \colon \overline{\mc{C}}_{g,n}(\P^1,d) \to \M_{g,n}(\P^1,d)$ is the universal curve and $f \colon \overline{\mc{C}}_{g,n}(\P^1,d) \to \P^1$ is the universal morphism. We exploit this relation and the standard torus action on $\P^1$, to compute the spin GW theory of $(\P^1,\mc{O}(-1))$ using the localisation formula of Graber--Pandharipande \cite{GP99}. We consider the following equivariant integrals of descendent classes of $(\P^1,\mc{O}(-1))$:
\begin{equation}
	\Braket{ \tau_{k_1}(\gamma_1) \cdots \tau_{k_n}(\gamma_n) }_{g,d}^{\P^1,\mc{O}(-1)}
	\coloneqq
	\int_{[\M_{g,n}(\P^1,d)]^{\loc,\mc{O}(-1)}} \prod_{i=1}^n \psi_i^{k_i} \ev_i^{\ast}(\gamma_i)
\end{equation}
with $\gamma_i \in H^{*}_{\C^*}(\P^1)$. Define the equivariant tau-function by:
\begin{equation}\label{def:EquivTauFn}
	\tau(x,x^{\star};u,q)
	\coloneqq
	\sum_{g,d} (-8u)^{g-1} (-4q)^{d}
	\Braket{
		\exp \left( 2 \sum_{k \ge 0} \bigl( x_k \tau_k(\bm{0}) + x_k^{\star} \tau_k(\bm{\infty}) \bigr) \right)
	}_{g,d}^{\P^1, \mc{O}(-1)}
\end{equation}
where $\bm{0}$ and $\bm{\infty}$ are the the equivariant Poincaré duals of $0$ and $\infty$. Our main result is an expression of this function in the operator formalism.

\begin{introthm} \label{thm:sGW:as:vev:intro}
	The equivariant tau-function of $(\P^1,\mc{O}(-1))$ can be represented in the neutral fermion Fock space as a single vacuum expectation value:
	\begin{equation}
		\tau(x,x^{\star};u,q)
		=
		\Braket{
			e^{\sum_i x_i \ms{B}_i }
			e^{\alpha^B_1} \Bigl( \frac{q}{u} \Bigr)^H e^{\alpha^B_{-1}}
			e^{\sum_j x_j^{\star} \ms{B}^\star_j }
		} .
	\end{equation}
	The formalism of the neutral fermion Fock space will be recalled in~\cref{subsec:fock}, and the definition of the operators involved in this theorem are defined in \cref{eqn:B:boson} and \cref{eqn:straightB}. 
\end{introthm}

The classes $\bm{0}$ and $\bm{\infty}$ form a basis of the equivariant cohomology of $\P^1$ as a module over the equivariant cohomology of the point. Therefore, after taking the non-equivariant limit of $\tau$, this theorem allows us to compute the full spin GW theory of $\P^1$. 

\begin{remark}
	The spin GW theory of $\P^1$ plays a distinguished role in the GW theory of surfaces. Indeed, if $\tilde{X} \to X$ is a blow-up at a smooth point of $X$, then the spin GW invariants of $\P^1$ are the classical GW invariants of $\tilde{X}$ with the curve class given by a multiple of the exceptional divisor. Therefore, \cref{thm:sGW:as:vev:intro} provides the first case of an all-genera expression of a GW potential of a target of dimension greater than $1$. 
\end{remark}

\subsection{Spin GW/H correspondence}
\label{subsec:spin:GW:H:intro}

The proof of~\cref{thm:sGW:as:vev:intro} involves an alternative enumerative theory associated to a spin curve. {\em Spin Hurwitz numbers} enumerate branched covers of a base Riemann surface endowed with a spin structure. Each cover is weighted by a sign determined by its parity \cite{EOP08}. More precisely, if $f \colon C \to B $ is a branched cover with odd ramifications over a spin curve $(B,\vartheta)$, and therefore even ramification divisor $R_f$, then
\begin{equation}
	N_{f,\vartheta} \coloneqq f^{\ast}\vartheta \otimes \mc{O}(\tfrac{1}{2} R_f)
\end{equation}
is a well-defined spin structure on $C$. Spin Hurwitz numbers are then defined as
\begin{equation}
	H_d(B,\vartheta ; \mu^1, \dotsc, \mu^n )
	\coloneqq
	\sum_{[f \colon C \to B]} \frac{(-1)^{p(N_{f,\vartheta})}}{|\Aut(f)|} ,
	\qquad 
	p(N_{f,\vartheta}) \equiv h^0(C, N_{f,\vartheta}) \pmod{2} ,
\end{equation}
where the ramification locus of $f$ is given by the ramification profiles $\mu^i$---odd partitions of the degree $d$---over $n$ points fixed on the base. 
Like spin GW invariants, these numbers only depend on the genus of $B$ and the parity of $ \theta$ (see e.g. \cite{Gun16}).

Inspired by the work of Okounkov and Pandharipande \cite{OP06a, OP06b}, we consider the (spin analogue of the) $k$-th {\em completed cycle} \cite{Lee20,MMN20} as the following linear combination of odd partitions:
\begin{equation}
	\overline{c}_{k}= \sum_{\mu \textup{ odd}} \kappa_{k,\mu} \cdot \mu ,
\end{equation}
where the coefficients $\kappa_{k,\mu}$ are given in terms of hyperbolic generating series (see \cref{subsec:spin:HNs}). 

\begin{introconj}[{Spin Gromov--Witten/Hurwitz correspondence,~\cite{Lee20}}] \label{conj:full:spin:GW:H}
	For all $(B,\vartheta)$, $g,d$ and $k_1,\ldots,k_n$ such that $\sum_i k_i = g-1+d(1-g(B))+n$, we have:\footnote{Due to different conventions for the map $ \C \{ \mc{OP} \} \to \Gamma $ in \cite[equation (3.20)]{GKL21} and \cite[section 5]{Lee20}, the factors $ 2^{-k_j}$ of \cite[equation (5.1)]{Lee20} do not appear in this formula.}
	\begin{equation} \label{eqn:full:spin:GW:H}
		\Braket{ \tau_{k_1} \cdots \tau_{k_n} }_{g,d}^{B,\vartheta}
		=
		H_d\left(B,\vartheta; \frac{(-1)^{k_1} k_1! }{(2k_1)!} \overline{c}_{2k_1+1},\ldots, \frac{(-1)^{k_n} k_n! }{(2k_n)!} \overline{c}_{2k_n+1}\right),
	\end{equation}
	where the function $H_d$ is extended by linearity, and the left-hand side stands for the {\em stationary spin GW invariants} defined as
	\begin{equation}
		\Braket{ \tau_{k_1} \cdots \tau_{k_n} }_{g,d}^{B,\vartheta}
		\coloneqq
		\int_{[\M_{g,n}(B,d)]^{\loc,\theta}} \prod_{i=1}^n \psi_i^{k_i} \ev_i^{\ast}(\omega) ,
	\end{equation}
	where $\omega \in H^2(\P^1,\Z)$ is the Poincar\'e-dual class of a point. 
\end{introconj}

This conjecture generalises conjectures of Maulik--Pandharipande which have been checked for $n = 0$ and for $d = 1, 2$ (see e.g.~\cite{MP08,LP07,KL11a, KL11b}). By taking the non-equivariant limit of~\cref{thm:sGW:as:vev:intro} we obtain the spin Gromov--Witten/Hurwitz (GW/H for short) correspondence for $\P^1$.

\begin{introthm}\label{thm:spin:GW:H:P1:intro}
	\Cref{conj:full:spin:GW:H} holds for $(B,\vartheta) = (\P^1,\mc{O}(-1))$.
\end{introthm}

The classical GW/H correspondence was established through the combination of the equivariant theory of $\mathbb{P}^1$ and the degeneration formulas (for both GW and Hurwitz theories) that allow to compute the invariants of a target curve in terms of the invariants of a nodal degeneration of this curve (see \cref{fig:deg}). A degeneration formula holds for spin Hurwitz theory and satisfies a peculiar property: the ramification at the node of the degenerated curve should be even (i.e. the contact orders are odd). A similar property was proved to hold for double ramification cycles \cite{CSS21}. Lee--Parker conjectured that the spin GW theory satisfies the same property \cite{LP13}, namely

\begin{introconj} \label{conj:degen:sGW}
	There exist coefficients $\widetilde{\kappa}_{k,\mu} \in \Q$ for all positive odd integers $k$ and odd partitions $\mu$ satisfying the following properties.
	\begin{enumerate}
		\item
		If $|\mu|\geq k$, then $\widetilde{\kappa}_{k,\mu} = 0$ unless $\mu = (k)$, and in this case $\widetilde{\kappa}_{k,(k)} = 1$.

		\item
		For all spin curves $(B,\vartheta)$ and all $g,d,$ and $k_1,\ldots,k_n$ as in \cref{conj:full:spin:GW:H}, we have:
		\begin{equation}
			\Braket{ \tau_{k_1} \cdots \tau_{k_n} }_{g,d}^{\bullet,B,\vartheta}
			=
			\sum_{\substack{\mu^1,\ldots,\mu^n \vdash d \\ \textup{odd}}}
				H_d^\bullet(B,\vartheta;\mu^1,\ldots,\mu^n)
				\prod_{i=1}^n \frac{(-1)^{k_i} k_i! }{(2k_i)!} \widetilde{\kappa}_{2k_i+1,\mu^i} .
		\end{equation}
		(The $\bullet$ superscript on both sides stands for invariants of maps with possibly disconnected domains.)
	\end{enumerate}
\end{introconj}

\begin{figure}
	\centering
	\begin{tikzpicture}[x=1pt,y=1pt,scale=.65]
		\fill[fill=gray!10] (22.598, 512.208) .. controls (28.604, 515.942) and (34.602, 518.344) .. (40.592, 519.414) -- (18.2347, 516.93) -- cycle;
		\fill[fill=gray!10] (150.598, 512.208) .. controls (156.604, 515.942) and (162.602, 518.344) .. (168.592, 519.414) -- (146.2347, 516.93) -- cycle;
		\fill[fill=gray!10] (143.142, 464.215) .. controls (154.094, 461.478) and (165.047, 458.667) .. (176, 458.667) .. controls (197.333, 458.667) and (218.667, 469.333) .. (229.333, 485.333) .. controls (206.955, 469.549) and (178.2247, 462.5097) .. (143.142, 464.215) -- cycle;
		\draw[thick] (48, 565.3333) .. controls (26.6667, 565.3333) and (5.3333, 554.6667) .. (-5.3268, 538.6667) .. controls (-15.987, 522.6667) and (-15.974, 501.3333) .. (-5.3073, 485.3333) .. controls (5.3593, 469.3333) and (26.6797, 458.6667) .. (48.0065, 458.6667) .. controls (69.3333, 458.6667) and (90.6667, 469.3333) .. (112, 469.3333) .. controls (133.3333, 469.3333) and (154.6667, 458.6667) .. (176, 458.6667) .. controls (197.3333, 458.6667) and (218.6667, 469.3333) .. (229.3333, 485.3333) .. controls (240, 501.3333) and (240, 522.6667) .. (229.3333, 538.6667) .. controls (218.6667, 554.6667) and (197.3333, 565.3333) .. (176, 565.3333) .. controls (154.6667, 565.3333) and (133.3333, 554.6667) .. (112, 554.6667) .. controls (90.6667, 554.6667) and (69.3333, 565.3333) .. cycle;
		\draw[thick] (16, 520) .. controls (32, 496) and (64, 496) .. (80, 520);
		\draw[thick] (22.598, 512.208) .. controls (39.5327, 522.736) and (56.4027, 522.6763) .. (73.208, 512.029);
		\draw[thick] (144, 520) .. controls (160, 496) and (192, 496) .. (208, 520);
		\draw[thick] (150.598, 512.208) .. controls (167.5327, 522.736) and (184.4027, 522.6763) .. (201.208, 512.029);
		
		\node at (192, 552) {\small$\bullet$};
		\node at (32, 544) {\small$\bullet$};
		\node at (20, 575) {$k_1$};
		\node at (200, 575) {$k_n$};
		\node at (112, 575) {$\cdots$};

		\draw[->,thick,decorate,decoration = {snake,pre length=2pt,post length=2pt,}](272, 512) -- (320, 512);
		
		\fill[fill=gray!10] (390.598, 512.208) .. controls (396.604, 515.942) and (402.602, 518.344) .. (408.592, 519.414) -- (386.2347, 516.93) -- cycle;
		\fill[fill=gray!10] (518.598, 512.208) .. controls (524.604, 515.942) and (530.602, 518.344) .. (536.592, 519.414) -- (514.2347, 516.93) -- cycle;
		\fill[fill=gray!10] (511.142, 464.215) .. controls (522.094, 461.478) and (533.047, 458.667) .. (544, 458.667) .. controls (565.333, 458.667) and (586.667, 469.333) .. (597.333, 485.333) .. controls (574.955, 469.549) and (546.2247, 462.5097) .. (511.142, 464.215) -- cycle;
		\fill[fill=gray!10] (548.472, 569.6551) arc[start angle=-131.6664, end angle=-18.5935, radius=29.2196] .. controls (583.5884, 573.5943) and (567.8819, 569.4239) .. (548.472, 569.6551) -- cycle;
		\fill[fill=gray!10] (386.5474, 567.6855) arc[start angle=-121.1715, end angle=-17.9663, radius=26.0366] .. controls (414.125, 574.1695) and (401.3771, 569.4207) .. (386.5474, 567.6855) -- cycle;
		\draw[thick] (416, 565.3333) .. controls (394.6667, 565.3333) and (373.3333, 554.6667) .. (362.6732, 538.6667) .. controls (352.013, 522.6667) and (352.026, 501.3333) .. (362.6927, 485.3333) .. controls (373.3593, 469.3333) and (394.6797, 458.6667) .. (416.0065, 458.6667) .. controls (437.3333, 458.6667) and (458.6667, 469.3333) .. (480, 469.3333) .. controls (501.3333, 469.3333) and (522.6667, 458.6667) .. (544, 458.6667) .. controls (565.3333, 458.6667) and (586.6667, 469.3333) .. (597.3333, 485.3333) .. controls (608, 501.3333) and (608, 522.6667) .. (597.3333, 538.6667) .. controls (586.6667, 554.6667) and (565.3333, 565.3333) .. (544, 565.3333) .. controls (522.6667, 565.3333) and (501.3333, 554.6667) .. (480, 554.6667) .. controls (458.6667, 554.6667) and (437.3333, 565.3333) .. cycle;
		\draw[thick] (384, 520) .. controls (400, 496) and (432, 496) .. (448, 520);
		\draw[thick] (390.598, 512.208) .. controls (407.5327, 522.736) and (424.4027, 522.6763) .. (441.208, 512.029);
		\draw[thick] (512, 520) .. controls (528, 496) and (560, 496) .. (576, 520);
		\draw[thick] (518.598, 512.208) .. controls (535.5327, 522.736) and (552.4027, 522.6763) .. (569.208, 512.029);
		\draw[thick] (400.0244, 589.9635) circle[radius=26.0366];
		\draw[thick] (567.8966, 591.483) circle[radius=29.2196];
		\node at (568, 616) {\small$\bullet$};
		\node at (392, 608) {\small$\bullet$};
		\node at (561.859, 562.894) {\small$\bullet$};
		\node at (404.0586, 564.2414) {\small$\bullet$};
		\node at (371, 615) {$k_1$};
		\node at (575, 630) {$k_n$};
		\node at (480, 600) {$\cdots$};
		\node at (408, 552) {$\mu^1$};
		\node at (550, 550) {$\mu^n$};
	\end{tikzpicture}
	\caption{The degeneration of the spin curve $(B,\vartheta)$ motivating \cref{conj:degen:sGW}.}
	\label{fig:deg}
\end{figure}

Assuming the above degeneration formula, one can deduce the full spin GW/H correspondence from the spin equivariant theory of $\P^1$.

\begin{introthm}\label{thm:full:spin:GW:H}
	\Cref{thm:spin:GW:H:P1:intro} and \cref{conj:degen:sGW} imply \cref{conj:full:spin:GW:H} in full generality.
\end{introthm}

\subsection{Algorithmic computation, closed formulae, and 2-BKP}

As an application, our methods provide an algorithm to explicitly compute any spin GW invariant of the Riemann sphere. The first closed formulae in this direction, for disconnected invariants without degree-zero components (denoted here with a $\emptyset$ subscript), appear in a conjecture of Maulik--Pandharipande \cite{MP08} for degree one and two, proved in \cite{KL11b, Lee13}. From degree three onward, the tail of completed cycles starts playing a role, and the formula is expected to acquire more terms. We show these terms and how to compute them. First, we find it more natural to gather these invariants in generating series:
\begin{equation}
	\Braket{ \tau_{k_1} \cdots \tau_{k_n} }_{\emptyset,g,d}^{\bullet, \P^1, \mc{O}(-1)}
	=
	\frac{2^d}{(d!)^2}
	\prod_{i=1}^n (-2)^{-k_i} k_i! \bigl[ z_i^{2k_i+1} \bigr] \mc{U}_d(z_1, \dots, z_n) .
\end{equation}
The expression of \cref{thm:sGW:as:vev:intro} gives an explicit algorithm for computing the generating series $\mc{U}_d$. The first degrees are given by
\begin{equation}\label{eqn:Ud}
\begin{split}
	\mc{U}_1(z_1, \dots, z_n) &= \frac{1}{2}\prod_{i=1}^n \sinh(z_i) ,
	\\
	\mc{U}_2(z_1, \dots, z_n) &= \frac{1}{2}\prod_{i=1}^n \sinh(2z_i) ,
	\\
	\mc{U}_3(z_1, \dots, z_n) &= \frac{1}{2} \prod_{i = 1}^n \sinh (3z_i) +\frac{1}{4} \prod_{i =1}^n \big(\sinh (2z_i)+ \sinh (z_i) \big) .
\end{split}
\end{equation}
The formulae for $\mc{U}_1$ and $\mc{U}_2$ reprove the Maulik--Pandharipande conjecture as~\cref{eqn:Ud} may be re-written as:
\begin{equation}
\begin{split}
	\Braket{\tau_{k_1} \cdots \tau_{k_n}}^{\bullet, \P^1, \mc{O}(-1)}_{\emptyset,g,1}
	=
	\prod_{i=1}^{n}\frac{k_i!}{(2k_i+1)!}(-2)^{-k_i} ,
	\\
	\Braket{\tau_{k_1} \cdots \tau_{k_n}}^{\bullet, \P^1, \mc{O}(-1)}_{\emptyset,g,2}
	=
	\frac{1}{2}\prod_{i=1}^{n}\frac{2 \cdot k_i!}{(2k_i+1)!}(-2)^{k_i} .
\end{split}
\end{equation}
The formula for $\mc{U}_3$ can be expanded into a new closed-form expression for degree three invariants: 
\begin{equation}
	\Braket{ \tau_{k_1} \cdots \tau_{k_n} }_{\emptyset,g,3}^{\bullet,\P^1, \mc{O}(-1)}
	=
	\frac{1}{9} \prod_{i=1}^n \frac{3 \cdot k_i!}{(2k_i+1)!} \Big( -\frac{9}{2} \Big)^{k_i}
	+ \frac{1}{18} \prod_{i=1}^n \frac{k_i!}{(2k_i+1)!} \big( (-2)^{-k_i} + 2 (-2)^{k_i} \big) .
\end{equation}
With the same methods, a formula for $n=1$ and recursive in the degree $d$ is derived. The recursion can compute high degree values in a short amount of time (it takes less than a second to compute up to $d = 15$ on an ordinary computer). The first generating series are shown in \cref{table:Ud:1pnt}.

\begin{table}
	\centering
	\begin{tabular}{c | l}
		\toprule
		$d$ & $\mc{U}_d(z)$
		\\
		\midrule
		$1$ & $\frac{1}{2} \sh_1$ 
		\\[1ex]
		$2$ & $\frac{1}{2} \sh_2$ 
		\\[1ex]
		$3$ & $\frac{1}{2} \sh_3 + \frac{1}{4} \sh_2 + \frac{1}{4} \sh_1$ 
		\\[1ex]
		$4$ & $\frac{1}{2} \sh_4 + \sh_3 + \sh_1$ 
		\\[1ex]
		$5$ & $\frac{1}{2} \sh_5 + \frac{9}{4} \sh_4 + \sh_3 + \sh_2 + \frac{9}{4} \sh_1$ 
		\\[1ex]
		$6$ & $\frac{1}{2} \sh_6 + 4 \sh_5 + \frac{25}{4} \sh_4 + \frac{1}{2} \sh_3 + \frac{27}{4} \sh_2 + \frac{9}{2} \sh_1$ 
		\\[1ex]
		$7$ & $\frac{1}{2} \sh_7 + \frac{25}{4} \sh_6 + \frac{81}{4} \sh_5 + \frac{99}{8} \sh_4 + \frac{25}{4} \sh_3 + \frac{211}{8} \sh_2 + \frac{99}{8} \sh_1$ 
		\\[1ex]
		$8$ & $\frac{1}{2} \sh_8 + 9 \sh_7 + 49 \sh_6 + 81 \sh_5 + 18 \sh_4 + 67 \sh_3 + 81 \sh_2 + 59 \sh_1$
		\\[1ex]
		\bottomrule
	\end{tabular}
	\caption{The generating series of $1$-point, degree $d$, stationary spin GW invariants of $(\P^1,\mc{O}(-1))$. Here $\sh_d = \sinh(dz)$.}
	\label{table:Ud:1pnt}
\end{table}

One main advantage of \cref{thm:sGW:as:vev:intro}, the expression of equivariant spin GW invariants in the Fock space formalism, is that this formalism is very well-adapted to the theory of integrable hierarchies. Integrable hierarchies are infinite collections of compatible partial differential equations for a single function in infinitely many variables. A famous example is the Korteweg--de Vries (KdV) hierarchy, a special case of the Kadomtsev--Petviashvili (KP) hierarchy, which by the Kontsevich--Witten theorem \cite{Kon92,Wit91} describes intersection theory on $\M_{g,n}$, i.e. GW theory of a point. Okounkov--Pandhari\-pande \cite{OP06b} applied their expression for the equivariant GW theory of $\P^1$ to prove it satisfies another hierarchy, called the 2D Toda lattice hierarchy. This is equivalent to two coupled KP hierarchies, labelled by another discrete lattice parameter.

Our expression in the neutral fermion Fock space allows for a similar result. Because we work with the neutral fermions, we obtain a B-type version of the hierarchy, known as the $2$-BKP hierarchy. As its name suggests, this hierarchy governs two coupled B-type KP hierarchies. The lattice parameter, however, is not available in this context.

\begin{introthm}\label{thm:2BKP}
	After a triangular linear change of variables $\{x_i,x_i^{\star}\} \mapsto \{t_i,s_i\}$, the equivariant tau-function $\tau$ satisfies the $2$-BKP hierarchy.
\end{introthm}

For some background on the 2-BKP hierarchy, see \cref{subsec:2-BKP}.

\subsection{Strategy of proof of~\texorpdfstring{\cref{thm:sGW:as:vev:intro}}{the main theorem}}

The spin Hurwitz numbers play a central role in the proof of \cref{thm:sGW:as:vev:intro}, as we will follow a parallel strategy to \cite{OP06b} to compute the equivariant spin GW theory of $\P^1$. In the ordinary (i.e. non-spin) setting, Okounkov and Pandharipande establish the bridge between GW theory and Hurwitz numbers by virtual localisation on the space of stable maps to equivariant $\P^1$. The integrals appearing in the localisation formula match with the right-hand side of the celebrated ELSV formula \cite{ELSV01}:
\begin{equation}
	h_{g;\mu}
	=
	\left( \prod_{i=1}^n \frac{\mu_i^{\mu_i}}{\mu_i!} \right)
	\int_{\M_{g,n}} \frac{\Lambda (-1)}{\prod_{i=1}^n (1-\mu_i \psi_i)}
	.
\end{equation}
The left-hand side of this equation is an ordinary Hurwitz number with one fixed ramification profile $\mu$ and simple ramification elsewhere. From this match, they obtain an expression for GW invariants which is \emph{quadratic} in the Hurwitz numbers. Exploiting the Fock space representation of ordinary Hurwitz numbers, they proceed to the ordinary analogue of \cref{thm:sGW:as:vev:intro}, which in turn establishes the integrability of the equivariant GW theory with respect to the $2$D Toda lattice hierarchy of PDEs and the ordinary GW/H correspondence for $\P^1$ by taking the non-equivariant limit.

To implement this plan in the spin setting, a desired ingredient is an ELSV-type formula for the simplest instance of completed cycles spin Hurwitz numbers, which needs to be compatible with the virtual localisation of the localised fundamental class. The $(r+1)$-completed cycles spin Hurwitz numbers are obtained by imposing that all ramifications except the one over zero are given by completed cycles of the same type:
\begin{equation}
	h^{+,r}_{g;\mu}
	\coloneqq
	\frac{|\Aut{(\mu)}|}{b!} H_d\bigl( \P^1, \mc{O}(-1); \mu, (\bar{c}_{r+1})^b \bigr) ,
\end{equation}
where the Riemann--Hurwitz theorem imposes $rb = 2g-2+ \ell (\mu) + d$.

Spin Hurwitz numbers do not allow even-order ramification index. Hence, simple ramifications must be skipped as base case, and the first natural oddly ramified instance is given by the $3$-completed cycles. The presence of completed cycles for even the simplest case is in contrast to the non-spin case, which features simple ramifications. This complicates the operator formalism considerably, and new strategies must be adopted (especially in \cref{sec:dH:operator:form}) to achieve the desired results. An ELSV-type formula for $3$-completed cycles spin Hurwitz numbers appears in \cite[equation~(9.24)]{GKL21} and features a double Hodge class as integrand\footnote{What appears op. cit. is a complete family of ELSV formulae for $(r+1)$-completed cycles spin Hurwitz number for any even $r$. Here we only recall the one for $r = 2$, which is much simpler than the general one. This family of ELSV formulae is proposed and proved to be equivalent to the Eynard--Orantin topological recursion for the corresponding Hurwitz numbers in \cite{GKL21}. Then, Alexandrov--Shadrin \cite{AS23} proved the topological recursion statement, hence proving the ELSV formulae under discussion for any $r$.}:
\begin{equation}
	h^{+,2}_{g;\mu}
	=
	2^{g-1+n}
	\left( \prod_{i=1}^n \frac{\mu_i^{\frac{\mu_i - 1}{2}}}{\bigl(\frac{\mu_i - 1}{2}\bigr)!} \right)
	\int_{\M_{g,n}} \frac{ \Lambda (2) \Lambda (-1)}{\prod_{i=1}^n (1-\mu_i \psi_i)}
	.
\end{equation}
The virtual localisation analysis of the equivariant spin GW invariants of $\P^1$ with its unique spin structure $\mc{O}(-1))$ leaves some freedom in lifting the equivariant structure to $\mc{O}(-1)$: the weights of the representation at the fibres are only restricted to differ by $1$. We use polynomiality of the localised class to make a (non-standard but symmetric) choice of weights, which shows that the double Hodge class in the above ELSV-type formula is exactly what is needed for the match: one Hodge class arises from the virtual fundamental class, the other from the top Chern class appearing in \cref{eqn:locsplit}.

Using this ELSV-type formula to write the spin GW invariants as a quadratic expression in spin Hurwitz numbers, we can then re-express them as a product of vacuum expectation values in the neutral fermion Fock space. That shape can then be manipulated to obtain a single vacuum expectation value, which reduces to spin Hurwitz numbers in the non-equivariant limit to establish~\cref{thm:spin:GW:H:P1:intro}.

The main technical step in the fermion analysis concerns the extension of the vacuum expectation value of the Hodge integrals to arbitrary complex $\mu_i$, and the calculation of the commutators of the required operators. In the ordinary case, Okounkov--Pandharipande~\cite{OP06a} employed a sophisticated analysis involving hypergeometric functions to obtain these results. We rather follow the alternative route followed by Oblomkov--Okounkov--Pandhari\-pande~\cite{OOP20}, by finding an ODE for a certain dressing operator.

\begin{remark}
	There is a large literature on simple and triple Hodge integrals as they relate to the theory of curves (via the above ELSV formula), and to the local 3-folds theory of curves (see for instance \cite{MV01,LLZ02,BP08}). However, few results are known for double Hodge integrals that play a central role in the present paper, apart from the integral of $\lambda_{g}\lambda_{g-1}$ with product of $\psi$-classes computed in~\cite{FP00}. In particular, Mumford's relation allows to consider simple (resp. double) Hodge integrals as a specialisation of triple (resp. quadruple) Hodge integrals. Thus the integrals studied in the present paper should be studied within the context of local 4-folds (see for instance \cite{KP08,CK20}). It is remarkable that they appeared recently in a seemingly unrelated problem: Blot conjectured that these integrals are used to compute the quantisation of the GW theory of the point by so-called double ramification cycles \cite{Blo22}.
\end{remark}

\subsection{Outline of the paper} 
\label{subsec:outline}
\Cref{sec:background} contains the necessary background, especially the technical definitions of bosonic operators built out of neutral fermions which are required to express spin Hurwitz numbers as vacuum expectations. In \cref{sec:localisation} we perform the virtual localisation of the localised virtual class and match it with the double Hodge ELSV formula for spin Hurwitz numbers. In \cref{sec:dH:operator:form} we extend the vacuum expectation expression as a generating series over the complex numbers. In \cref{sec:sGW:operator:form} we conclude the proof of the spin GW/H correspondence for the Riemann sphere, discuss the spin GW/H correspondence for a general target, and report on the explicit computations presented in the introduction. In \cref{sec:2BKP} we prove string and divisor equations and prove $2$-BKP for the equivariant tau-function.

\addtocontents{toc}{\protect\setcounter{tocdepth}{1}}
\subsection*{Notation} 
\label{subsec:notation}
We use the functions
\begin{equation}\label{eqn:fnctns:notation}
	\varsigma (z) = 2\sinh \left(\frac{z}{2}\right) ,
	\qquad 
	\mc{S}(z) = \frac{\sinh (\frac{z}{2})}{(\frac{z}{2})} ,
	\qquad 
	\qoppa(z) = \frac{\cosh(\frac{z}{2})}{2} .
\end{equation}
We also denote by $\mc{P}$ the set of \emph{partitions}, by $\SP$ the set of \emph{strict partitions} (i.e. all parts are distinct), and by $\OP$ the set of \emph{odd partitions} (i.e. all parts are odd). The sets of partitions of a given integer $d$ carry a subscript $d$. In particular, $ \mc{P}_0 = \mc{OP}_0 = \mc{SP}_0 $ contains the empty partition.

For any complex number $z \in \C$, we set
\begin{equation}
	z' \coloneqq \frac{z-1}{2}.
\end{equation}

For GW invariants, Hurwitz numbers, and moduli spaces, we use a superscript $ \bullet $ to indicate the disconnected kind, and add a subscript $ \emptyset$ to exclude degree zero components.\footnote{Note that several papers on spin GW theory, e.g. \cite{LP07,MP08,Lee13,LP13,Lee20}, use a superscript $ \bullet $ to denote `disconnected without degree zero components'.} For degree $0$, these disconnected counts contain the map or cover from the empty curve to the target. Connected versions either do not carry a superscript or, for clarity, may have a superscript $ \circ$.

\subsection*{Acknowledgments.}
\label{subsec:ack}
The authors would like to thank John Harnad, Jan-Willem van~Ittersum, Felix Janda, Martijn Kool, Junho Lee, Sergej Monavari, Rahul Pandharipande, and Sergey Shadrin for useful discussions, and Lou-Jean Leila Cobigo and Marvin Anas Hahn for pointing out an error in our formula for spin-completed cycles.

R.~K.~has been supported by the Max-Planck-Gesellschaft, by the Natural Sciences and Engineering Research Council of Canada, and by the Pacific Institute for the Mathematical Sciences (PIMS). The research and findings may not reflect those of these institutions. The University of Alberta respectfully acknowledges that we are situated on Treaty 6 territory, traditional lands of First Nations and M\'{e}tis people. A.~G. and D.~L. have been supported by the ERC-SyG project ``Recursive and Exact New Quantum Theory'' (ReNewQuantum) which received funding from the European Research Council (ERC) under the European Union’s Horizon 2020 research and innovation programme under grant agreement No 810573, and by the grant host institutes IHES and IPhT. D.~L. is funded by Swiss National Foundation Ambizione project ``Resurgent topological recursion, enumerative geometry and integrable hierarchies'' under the grant agreement PZ00P2-202123, and by the INdAM group GNSAGA. The authors also thank the University of Geneva, IPhT, IHES, and Universiteit Leiden for providing great working conditions during the collaborations that led to this work.

\addtocontents{toc}{\protect\setcounter{tocdepth}{2}}
\newpage

\section{Neutral fermions and spin Hurwitz numbers}
\label{sec:background}

In this section, we recall some facts about the neutral fermion formalism and its application to spin Hurwitz numbers. This material is mostly adapted from \cite{GKL21}.

\subsection{Neutral fermion Fock space}
\label{subsec:fock}

\begin{definition}\label{def:neutral:fermions}
	Let $W^B$ be the infinite-dimensional complex vector space freely generated by $ \{ \phi_k\}_{k \in \Z} $. Define the bilinear form
	\begin{equation}
		(\phi_k,\phi_l) = \frac{(-1)^k}{2} \delta_{k+l}
	\end{equation}
	and the space of \emph{neutral fermions} as the associated Clifford algebra $\Cl^B \coloneqq \Cl(W^B,(\cdot,\cdot))$. It has a $\Z/2\Z$ grading $\Cl^B_0 \oplus \Cl^B_1$, with $\Cl^B_{p}$ spanned by products of $m$ elements with $m \equiv p \pmod{2}$. Moreover, it has canonical anticommutation relations given by
	\begin{equation}\label{eqn:CAR:B}
		\{\phi_k,\phi_l\} = (-1)^k\delta_{k+l} .
	\end{equation}
\end{definition}

\begin{definition}\label{def:Fock:space:B}
	Consider the subspace $L^B$ of $W^B$ generated by $\phi_k$ for $k < 0$, which is maximal isotropic for $(W^B,(\cdot,\cdot))$. Define the \emph{(fermionic) Fock space of type B} as the unique graded highest-weight left module of $\Cl^B$:
	\begin{equation}
		\mf{F}^B = \Cl^B / \Cl^B \cdot L^B .
	\end{equation}
	We write $\ket{0}$ for the class of $1$, also called the \emph{vacuum state}, and $\ket{1} = \sqrt{2} \phi_0 \ket{0}$. The space $\mf{F}^B$ inherits the $\Z/2\Z$ grading: $\mf{F}^B = \mf{F}^B_0 \oplus \mf{F}^B_1$, and $\ket{p} \in \mf{F}^B_{p}$. Moreover, for $\lambda \in \SP$, define
	\begin{equation}
		\ket{\lambda} = \phi_{\lambda_1} \dotsb \phi_{\lambda_{\ell (\lambda )}} \ket{p(\lambda)} ,
		\qquad \text{where} \qquad p(\lambda) \equiv \ell(\lambda) \pmod{2}
	\end{equation}
	is the \emph{parity} of $\lambda$. In particular, $\{ \ket{\lambda} \mid \lambda \in \SP \}$ form a basis of $\mf{F}^B_0$.

	With the dual construction, (i.e. considering the unique graded highest-weight right module) we define the dual Fock space $\mf{F}^{B,\ast}$ and the covectors $\bra{0}$ and $\bra{1}$. In particular, we have a pairing $\mf{F}^{B,\ast} \times \mf{F}^B \to \C$ denoted by
	\begin{equation}
		\braket{\xi | \omega} = \bigl( \bra{\xi},\ket{\omega} \bigr) .
	\end{equation}
	For any $O \in \Cl^B$ we can define its \emph{vacuum expectation value} $\braket{O}$ as $\braket{0| O | 0}$. Since the (right) action of $\mc{B}$ on the dual Fock space is the adjoint of the (left) action on the Fock space, there is no ambiguity in the notation.
\end{definition}

\begin{lemma}\label{lem:VEV:Fock:basis:B}
	The vacuum expectation values of quadratic expressions in the $\phi$'s are
	\begin{equation}
		\braket{ \phi_k \phi_l } = (-1)^k \delta_{k+l} u[l] ,
		\qquad \text{where} \qquad 
		u[l] =
		\begin{cases}
			1 & \text{if $l > 0$}, \\
			\frac{1}{2} & \text{if $l = 0$}, \\
			0 & \text{if $l < 0$}.
		\end{cases}
	\end{equation}
\end{lemma}

In the following, we will mainly consider the sector $\mf{F}^B_0$ of the Fock space. To start with, let us recall the action by an infinite Lie algebra, which corresponds to the infinite-dimensional analogue of the orthogonal Lie algebra $\mf{o}_{2n+1}$. The choice of convention for the basis elements is taken from \cite{Gia21}.

\begin{definition}\label{def:b:inf}
	Let $\mf{a}_\infty = \set{ (a_{m,n})_{m,n \in \Z} | a_{m,n} = 0 \text{ for } |m-n| \gg 0 }$ be the bi-infinite general linear Lie algebra with the standard commutator bracket of band matrices. It has a natural basis of elementary matrices $\set{ E_{j,k} = (\delta_{m,j}\delta_{n,k})_{m,n} }_{j,k \in \Z}$. Define the involution $\iota \colon E_{j,k} \mapsto (-1)^{j+k} E_{-k,-j}$ and the Lie subalgebra
	\begin{equation}
		\mf{b}_\infty \coloneqq \bigl\{ g \in \mf{a}_\infty \bigm| \iota (g) = -g\bigr\} .
	\end{equation}
	It has a basis $\set{ B_{j,k} = (-1)^k E_{j,k} - (-1)^{j} E_{-k,-j} }_{j>k}$ and two irreducible representations on $\mf{F}^B_0$: a linear, $ \pi$, and a projective, $ \hat{\pi}$, given on this basis by
	\begin{equation}
		\pi \colon B_{j,k} \longmapsto \phi_j \phi_{-k} ,
		\qquad\qquad
		\hat{\pi} \colon B_{j,k} \longmapsto \normord{\phi_j \phi_{-k}} .
	\end{equation}
	We write $\hat{\mf{b}}_\infty$ for the central extension of which $\hat{\pi}$ is a representation. From now on, we will write $ E^B_{j,k}= \phi_j \phi_{-k}$ and denote $ \hat{E}^B_{j,k} = \normord{\phi_j \phi_{-k}} $. Notice that $\hat{E}^B_{j,k} = - \hat{E}^B_{-k,-j}$.
\end{definition}

There are two main families of elements of $\hat{\mf{b}}_\infty$: the bosonic and the completed cut-and-join operators. Moreover in \cite{GKL21} the authors introduced a new algebra of operators that interpolates between the two families. In the non-spin case, such an algebra was introduced by Bloch--Okounkov~\cite{BO00} and packed into generating functions by Okounkov--Pandharipande in \cite{OP06a}.

\begin{definition}\label{def:B:operators}\leavevmode
	\begin{itemize}
		\item
		For any odd integer $m$, define the \emph{bosonic operators} $\alpha^B$ as:
		\begin{equation}\label{eqn:B:boson}
			\alpha^B_m
			\coloneqq
			\frac{1}{2} \sum_{k \in \Z} (-1)^k \hat{E}^B_{k-m,k}
			=
			\sum_{k > m/2} (-1)^k E^B_{k-m,k} .
		\end{equation}
		The operators $\alpha^B_m$ could also be defined for even arguments, but those would vanish. The bosonic operators form a subalgebra, called the \emph{Heisenberg algebra} $\mf{H}^B$, with canonical commutation relations
		\begin{equation}
			\bigl[ \alpha^B_m,\alpha^B_n \bigr] = \frac{m}{2}\delta_{m+n} .
		\end{equation}

		\item
		For any positive even integer $r$, define the \emph{completed cut-and-join operators} $\mc{F}^B$ as: 
		\begin{equation}\label{eqn:B:power:sum}
			\mc{F}^B_{r+1}
			\coloneqq
			\frac{1}{2} \sum_{k \in \Z} (-1)^k k^{r+1} \hat{E}^B_{k,k}
			=
			\sum_{k >0} (-1)^k k^{r+1} E^B_{k,k} .
		\end{equation}
		Again, the operators $\mc{F}^B_{r+1}$ could also be defined for odd $r$, but those would vanish.

		\item
		For any integer $m$, define the \emph{B-Okounkov--Pandharipande operators} as:
		\begin{equation}\label{eqn:B:OP}
			\mc{E}^{B}_m(z)
			\coloneqq
			\frac{1}{2} \sum_{k \in \Z} (-1)^k e^{(k+\frac{m}{2})z} \hat{E}^B_{k-m,k}
			+
			\delta_m \frac{\qoppa(z)}{\varsigma(z)} .
		\end{equation}
		The identity term involves the functions defined in \labelcref{eqn:fnctns:notation}. We also define the ``non-corrected'' operators by omitting the constant term:
		\begin{equation}
			\hat{\mc{E}}^{B}_m(z)
			\coloneqq
			\frac{1}{2} \sum_{k \in \Z} (-1)^k e^{(k+\frac{m}{2})z} \hat{E}^B_{k-m,k} .
		\end{equation}
		The bosonic and the completed cut-and-join operators are given by
		\begin{equation}
			\alpha_m^B = \mc{E}^{B}_m(0),
			\qquad\quad
			\mc{F}^B_{r+1} = (r+1)! [z^{r+1}] \hat{\mc{E}}^B_0(z) .
		\end{equation}
	\end{itemize}
\end{definition}

The bosonic operators naturally appear in the boson-fermion correspondence. Indeed, the Heisenberg algebra $\mf{H}^B$ has a unique (up to unique isomorphism) irreducible representation, given by $\Gamma = \C[t_1,t_3,\dots]$, with action $ \rho( \alpha^B_m ) = m\frac{\del}{\del t_m} $ and $ \rho (\alpha^B_{-m} ) = \frac{1}{2} t_m $ (for $ m $ a positive odd integer).

\begin{theorem}[{Boson-fermion correspondence, \cite{You89}}]\label{thm:BF:corresp}
	As a representation of $\mf{H}^B$, both $\mf{F}^B_0$ and $\mf{F}^B_1$ are irreducible, and therefore isomorphic to $\Gamma$. Explicitly,
	\begin{equation}
	\begin{aligned}
		\Phi^B \colon \mf{F}^B & \longrightarrow \Gamma[\xi]/(\xi^2 - 1) \\
		\ket{\omega} & \longmapsto \Braket{0 | \Gamma_+^B(t) | \omega} + \Braket{1 | \Gamma_+^B(t) | \omega} \xi
	\end{aligned}
	\end{equation}
	is an isomorphism. Here $\Gamma_+^B(t) = \exp(\sum_{k = 0}^{\infty} t_{2k+1} \alpha_{2k+1}^B )$ is called the vertex operator.
\end{theorem}

Via the boson-fermion correspondence, one can also compute the action of the bosonic and completed cut-and-join operators on the fermionic Fock space as follows.

\begin{proposition}\label{prop:boson:CnJ:action}
	The following results hold.
	\begin{itemize}
		\item
		For an odd partition $\mu \in \OP_d$, set $\alpha^B_{\pm\mu} = \alpha^B_{\pm\mu_1} \cdots \alpha^B_{\pm\mu_n}$. Then
		\begin{equation}
			\alpha^B_{-\mu} \ket{0}
			=
			\sum_{\lambda \in \SP_d} 2^{- \frac{p(\lambda)}{2} - \ell(\mu)} \zeta_\mu^\lambda \ket{\lambda} ,
			\qquad\quad
			\alpha^B_{\mu} \ket{\lambda} = 2^{- \frac{p(\lambda)}{2} - \ell(\mu)} \zeta_\mu^\lambda \ket{0} .
		\end{equation}
		where $\zeta_\mu^\lambda$ are the characters of the Sergeev group.

		\item
		For any strict partition $\lambda \in \SP$ and even positive integer $r$, the action of the completed cut-and-join operator $\mc{F}^B_{r+1}$ of type B is given by
		\begin{equation}
			\mc{F}^B_{r+1} \ket{\lambda} = p_{r+1}(\lambda) \ket{\lambda} ,
		\end{equation}
		where $p_{r+1}$ is the usual symmetric power-sum of odd index.
	\end{itemize}
\end{proposition}

In the next sections, we will also make use of some useful properties of the B-Okounkov--Pandharipande operators. Once again, the properties involves the functions \labelcref{eqn:fnctns:notation}. We recall them here for the reader's convenience, and refer to \cite{GKL21} for a proof.

\begin{proposition}\label{prop:OP:VEV}
	The B-Okounkov--Pandharipande operators satisfy the following properties.
	\begin{itemize}
		\item
		\textup{\textsc{Vacuum expectation values:}}
		\begin{equation}
			\Braket{ \mc{E}^{B}_m(z) }
			=
			\delta_m \frac{\qoppa(z)}{\varsigma(z)}
			=
			\frac{\delta_m}{4} \coth \left( \tfrac{z}{2} \right) .
		\end{equation}

		\item
		\textup{\textsc{Parity relations:}}
		\begin{equation}
			\mc{E}^{B}_m(-z) = (-1)^{m+1} \mc{E}^{B}_m(z) .
		\end{equation}
		The same holds for the corrected operators.

		\item
		\textup{\textsc{Commutation relations:}} the subspace of $\mf{b}_\infty $ spanned by the coefficients $[z^k]\mc{E}^{B}_m(z)$ is a Lie subalgebra. Explicitly,
		\begin{equation}\label{eqn:spin:OP:comm}
			\bigl[ \mc{E}^{B}_m(z),\mc{E}^{B}_n(w) \bigr]
			=
			\frac{1}{2}
			\varsigma\bigl( \left| \begin{smallmatrix}
				m & z \\
				n & w
			\end{smallmatrix} \right| \bigr) \mc{E}^{B}_{n+m}(z+w)
			+ \frac{(-1)^n}{2}
			\varsigma\bigl(\left| \begin{smallmatrix}
				m & - z \\
				n & w
			\end{smallmatrix} \right| \bigr) \mc{E}^{B}_{n+m} (z-w) .
		\end{equation}
	\end{itemize}
\end{proposition}

Note that \cite{You89,Lee20,GKL21} give strong links between the A- and B-type theories.

\subsection{Spin Hurwitz numbers}
\label{subsec:spin:HNs}

As explained in the introduction, spin Hurwitz numbers are weighted counts of covers of a curve with a spin structure, with a sign given by the parity \cite{EOP08}. Spin structures can be pulled back along branched covers, as long as all ramifications are odd: in that case the ramification divisor is even.

\begin{definition}
	Let $ (B,\vartheta )$ be a spin curve and $f \colon C \to B $ a branched cover with only odd ramifications. Denote by $R_f$ its ramification divisor. Then the \emph{twisted pullback of $\vartheta$ along $f$} is the line bundle on $C$
	\begin{equation}
		N_{f,\vartheta} \coloneqq f^{\ast}\vartheta \otimes \mc{O}(\tfrac{1}{2} R_f) .
	\end{equation}
	It is a spin structure on $C$.
\end{definition}

\begin{definition}\label{def:general:spin:HNs}
	Let $ (B,\vartheta )$ be a spin curve, $d$ a non-negative integer, $ x_1, \dotsc, x_n \in B$, and $ \mu^1, \dotsc, \mu^n \in \OP_d$. The \emph{spin Hurwitz number} is defined as
	\begin{equation}
		H_d(B,\vartheta ; \mu^1, \dotsc, \mu^n )
		\coloneqq
		\sum_{[f \colon C \to B]} \frac{(-1)^{p(N_{f,\vartheta})}}{|\Aut(f)|} ,
	\end{equation}
	where the sum is over all isomorphism classes of connected ramified covers with ramification profile $\mu^i $ over $x_i$ and unramified anywhere else. As usual, when dealing with disconnected covers, we add a superscript $ \bullet$.
\end{definition}

For partitions $\mu^i$ of arbitrary size, one can extend the definition as follows: if at least one of the $\mu^i$ has size greater than $d$, the Hurwitz number vanishes; otherwise
\begin{equation}\label{eqn:H:def}
	H_d(B,\vartheta ; \mu^1, \dotsc, \mu^n )
	\coloneqq
	\left( \prod_{i=1}^n \binom{m_1(\widehat{\mu^i})}{{m_1(\mu^i)}} \right)
	H_d\bigl( B,\vartheta ; \widehat{\mu^1}, \dotsc, \widehat{\mu^n} \bigr) ,
\end{equation}
where $\widehat{\mu}$ is the partition of size $d$ obtained from $\mu$ by adding 1's, and $m_1(\mu)$ stands for the number of entries equal to $1$.

For any positive integer $k$, we introduce the spin analogue of the $k$-th {\em completed cycle} (see e.g. \cite{MMN20, GKL21}) as an element
\begin{equation}\label{eqn:completed:cycle}
	\overline{c}_{k}
	\coloneqq
	\sum_{\mu \in \OP} \kappa_{k,\mu} \cdot \mu \in \C\{ \OP \} ,
\end{equation}
defined in terms of representation theory of Sergeev groups. We give an explicit expression in the following lemma.

\begin{lemma}
	For all odd partitions $\mu=(\mu_1,\ldots,\mu_\ell)$,
	\begin{equation}\label{eqn:CompCyclesFormula}
		\kappa_{k,\mu}
		\coloneqq
		2^{-g} (k-1)! \frac{\prod_{i=1}^\ell \mu_i}{|\mu|!} \bigl[ z^{2g} \bigr]
		\mc{S}(2z)\mc{S}(z)^{|\mu|-2}
		\prod_{i=1}^\ell \mc{S}(\mu_i z) \,, \qquad 2g = k+1-|\mu|-\ell(\mu).
	\end{equation}
\end{lemma}

\begin{proof}
	This follows by combining \cite[Theorem 1.1]{LP13} with \cite[Definition 4.5 \& Proposition 6.1]{GKL21} to find that
	\begin{equation}
		H( B, \theta; \mu^1, \dotsc, \mu^n, \bar{c}_k) = \sum_\mu \Big( \prod_{i = 1}^{\ell (\mu)} \mu_i \Big)  \, |\Aut (\mu)| \, H(\P^1, \mc{O}(-1); \mu, \bar{c}_k) H(B, \theta; \mu^1, \dotsc, \mu^n, \mu)\,,
	\end{equation}
	and hence,
	\begin{equation}
		\begin{split}
			\kappa_{k,\mu} 
			&= 
			\Big( \prod_{i = 1}^{\ell (\mu)} \mu_i \Big) \, |\Aut (\mu)| \, H(\P^1, \mc{O}(-1); \mu, \bar{c}_k)
			\\
			&=
			\frac{\prod_{i=1}^{\ell (\mu)} \mu_i}{|\mu|!} h^{+,r,b=1}_{\mu,(1^{|\mu|})}
			\\
			&=
			\frac{2^{1-g}(k-1)!}{|\mu|!} [z^k] \Big\< \prod_{i=1}^{\ell (\mu)} \mc{E}^B_{\mu_i}(0) \mc{E}^B_0 (z) \mc{E}_{-1}(0)^{|\mu|} \Big\>
			\\
			&=
			\frac{2^{1-g}(k-1)!\prod_{i=1}^{\ell(\mu)}\mu_i}{|\mu|!} [z^{2g}] \qoppa (z) \mc{S}(z)^{|\mu|-1} \prod_{i=1}^{\ell (\mu)} \mc{S} (\mu_i)\,,
		\end{split}
	\end{equation}
	which is equal to the statement of the lemma, because $ 2 \qoppa (z) \mc{S}(z) = \mc{S}(2z)$.
\end{proof}
Notice that, since $\mu$ is an odd partition and the function $\mc{S}(z) = \frac{e^{z/2} - e^{-z/2}}{z}$ is even, $\kappa_{k,\mu} = 0$ unless $k$ is an odd integer. Moreover, as $\mc{S}(z) = 1 + O(z^2)$, one can easily prove that the coefficients $\kappa_{k,\mu}$ satisfy property (1) from \cref{conj:degen:sGW}: if $|\mu| \ge k$, then $\kappa_{k,\mu} = 0$ unless $\mu = (k)$, and in this case $\kappa_{k,(k)} = 1$.

We are interested in the case $(B,\vartheta) = (\P^1,\mc{O}(-1))$, only one generic ramification $\mu$, and all other ramifications given by completed cycles. For these specific spin Hurwitz numbers, let us introduce the following notation.

\begin{definition}\label{def:spin:HNs}
	Let $r$ be a positive even integer. The \emph{$(r+1)$-completed cycles spin single Hurwitz numbers} for genus $g$ and generic ramification $\mu \in \OP_d$ are defined by
	\begin{equation}
		h^{+,r}_{g;\mu}
		\coloneqq
		\frac{|\Aut{(\mu)}|}{b!} H_d \bigl( \P^1, \mc{O}(-1); \mu, (\bar{c}_{r+1})^b \bigr) ,
	\end{equation}
	where $b = \frac{1}{r}(2g-2+ \ell (\mu) + |\mu|)$ is needed from the Riemann--Hurwitz formula to obtain genus $g$ source curves. The above definition is obtained by multilinearity, combining \cref{eqn:H:def,eqn:completed:cycle}.
\end{definition}

The superscript $+$ in the notation refers to $ \mc{O}(-1)$ being an even spin structure on $ \P^1$.

\subsubsection{Spin Hurwitz numbers as vacuum expectation}
\label{subsec:sHN:as:vev}

The enumeration of Hurwitz covers of $(\P^1, \mc{O}(-1))$ is known to be equivalent to multiplication in the class algebra of the Sergeev group \cite{EOP08,Gun16}. In other words, spin Hurwitz numbers can be expressed in terms of characters of the Sergeev group. Comparing the character formula and the action of the bosonic and completed cut-and-join operators reported in \cref{prop:boson:CnJ:action}, one can express spin Hurwitz numbers as vacuum expectations in the Fock space of type B.

\begin{proposition}[{\cite{Lee20,GKL21}}] \label{prop:sHN:VEV}
	Disconnected spin Hurwitz numbers with completed cycles are given by
	\begin{equation}\label{eqn:sHN:VEV}
		h_{g;\mu}^{\bullet, +, r}
		=
		2^{1-g} [u^b]
		\Braket{
			e^{\alpha_{1}^{B}}
			e^{u \frac{\mc{F}^B_{r+1}}{r+1}}
			\left( \prod_{i=1}^{\ell(\mu)} \frac{\alpha^B_{-\mu_i}}{\mu_i} \right)
		} ,
	\end{equation}
	where $b = \frac{2g - 2 + |\mu| +n}{r}$ is the number of completed cycles ramifications in the cover.
\end{proposition}

We remark that the above vacuum expectation can be written in a more symmetric way:
\begin{equation}
	h_{g;\mu}^{\bullet, +, r}
	=
	2^{1-g} [u^b]
	\Braket{
		\prod_{i=1}^{\ell(\mu)}
		e^{\alpha_{1}^{B}}
		e^{u \frac{\mc{F}^B_{r+1}}{r+1}}
		\frac{\alpha^B_{-\mu_i}}{\mu_i} 
		e^{u \frac{\mc{F}^B_{r+1}}{r+1}}
		e^{-\alpha_{1}^{B}}
	} .
\end{equation}
We will also need an expression for a slightly different collection of Hurwitz numbers.

\begin{proposition}\label{prop:mixed:HN:VEV}
	Disconnected spin Hurwitz numbers with mixed completed cycles are given by
	\begin{equation}
		H^\bullet_d(\P^1, \mc{O}(-1); \bar{c}_{2k_1+1}, \ldots, \bar{c}_{2k_b+1} )
		=
		2^d
		\Braket{
			\frac{(\alpha^B_1)^d}{d!}
			\left(
				\prod_{j=1}^b \frac{(2k_j)!}{2^{k_j}} [z^{2k_j+1}]\hat{\mc{E}}^B_0(z)
			\right)
			\frac{(\alpha^B_{-1})^d}{d!}
		} .
	\end{equation}
\end{proposition}

\begin{proof}
	The proof is completely analogous to that of \cite[proposition 6.1]{GKL21}.
\end{proof}

\subsubsection{A spin ELSV formula}
\label{subsec:spin:ELSV}

The vacuum expectation formula allows for explicit computations of the correlation functions associated to spin Hurwitz numbers with completed cycles, which was proved to be equivalent to the Eynard--Orantin topological recursion by Alexandrov and Shadrin \cite{AS23}. Prior to their work, topological recursion for spin Hurwitz numbers was shown to be equivalent to an ELSV-type formula in \cite{GKL21}. We recall here the formula for $r = 2$ (i.e. spin $3$-completed cycles), which is the simplest instance of Hurwitz numbers with completed cycles in the spin case.

\begin{theorem}[{\cite{GKL21,AS23}}]\label{thm:spin:ELSV}
	Spin Hurwitz numbers with $3$-completed cycles are given by
	\begin{equation}\label{eqn:spin:ELSV}
		h^{+, 2}_{g;\mu}
		=
		2^{g-1+n}
		\left( \prod_{i=1}^n \frac{\mu_i^{\frac{\mu_i - 1}{2}}}{\bigl(\frac{\mu_i - 1}{2}\bigr)!} \right)
		\int_{\M_{g,n}} \frac{ \Lambda (2) \Lambda (-1)}{\prod_{i=1}^n (1-\mu_i \psi_i)}
		,
	\end{equation}
	where $\Lambda(t) = \sum_{i=0}^g t^i \lambda_i$ is the Chern polynomial of the Hodge bundle.
\end{theorem}

The ELSV formula for higher $r$, which involved a product of the Chiodo class and the Witten $2$-spin class, can be found in \cite{GKL21}.

It is useful to gather double Hodge integrals into generating series as follows.

\begin{definition}\label{def:dH:gen}
	Define the connected generating series for double Hodge integrals as
	\begin{equation}\label{eqn:dH:conn:gen}
		\dHconn_{g}(z)
		\coloneqq
		2^{2g - 2 +n}
		\int_{\M_{g,n}} \Lambda(2) \Lambda (-1) \prod_{i = 1}^n \frac{z_i}{1 - z_i \psi_i} .
	\end{equation}
	The $0$-point generating series, both stable and unstable, is set to be $\dHconn_{g}() = 0$. The $1$- and $2$-point unstable cases are defined as
	\begin{equation}\label{eqn:dH:unstable}
		\dHconn_0(z_1) \coloneqq \frac{1}{2z_1} ,
		\qquad\text{and}\qquad
		\dHconn_0(z_1,z_2) \coloneqq \frac{z_1z_2}{z_1+z_2} .
	\end{equation}
	Let us define the corresponding genus generating series and its possibly disconnected counterpart as 
	\begin{equation}
		\dHconn(z;u) \coloneqq \sum_{g=0}^{\infty} u^{g-1} \dHconn_{g}(z) ,
		\qquad\quad
		\dHdisc(z;u) \coloneqq \sum_{\mu \vdash n} \frac{1}{|\Aut(\mu)|} \prod_{k=1}^{\ell(\mu)} \dHconn\bigl( z_{\mu_k};u \bigr) .
	\end{equation}
\end{definition}

\subsection{The 2-BKP hierarchy}
\label{subsec:2-BKP}

The BKP hierarchy is a B-type version of the (A-type) KP hierarchy. By work of the Kyoto school, see e.g. \cite{KM81,SS83}, $\tau$ functions of the KP hierarchy are related to elements of an infinite Grassmannian and hence are naturally related to $ \mf{a}_\infty$. Hence, they can be expressed as vacuum expectation values in the (charged) fermionic Fock space.

In \cite{DKM81}, Date--Kashiwara--Miwa found another, similar hierarchy, which was studied more in \cite{DJKM82,JM83}. It is related to the algebra $ \mf{b}_\infty$, and hence called the BKP hierarchy.\footnote{This is the \emph{small} BKP hierarchy. There is another, large, BKP hierarchy introduced by Kac--Van de Leur~\cite{KVdL98}. The latter is obtained by realising that in the construction of charged fermions, one takes an infinite-dimensional vector space and lets $ \mf{a}_\infty $ act on $ V \oplus V^*$ with the natural inner product. Then this action can be extended to some $ \mf{b}_{2\infty}$.} This hierarchy has many similar properties to the KP hierarchy, and in particular can be studied via the neutral fermion formalism. Its tau-functions are related to an orthogonal Grassmannian, and in the neutral fermionic Fock space they can be expressed as
\begin{equation}
	\tau_M (t) = 
	\Braket{
		\Gamma_+^B (t) M
	}, 
	\qquad\quad  
	\Gamma_+^B(t) = \exp \bigg( \sum_{k =0}^\infty t_{2k+1} \alpha^B_{2k+1} \bigg),
\end{equation}
with $ M \in \widehat{\mathrm{B}}_{\infty} \coloneqq \set{ e^{g_1} \cdots e^{g_N} | g_i \in \widehat{\mf{b}}_{\infty} } $. In other words, BKP tau-functions are the bosonic analogue through \cref{thm:BF:corresp} of elements in the orbit $\widehat{\mathrm{B}}_{\infty} \ket{0}$.

Like the KP hierarchy, the BKP hierarchy also has extensions to multiple components. Tau-functions of $2$-BKP are of the form \cite[page~5]{Tak06} (or \cite[equation~(4.9)]{Orl03} for the hypergeometric case),
\begin{equation}\label{eq:2BKP:tau}
	\tau_M (t,s) = 
	\Braket{
		\Gamma^B_+(t) M \Gamma_-^B(s)
	}, 
	\qquad\quad
	\Gamma_\pm^B(t) = \exp \bigg( \sum_{k =0}^\infty t_{2k+1} \alpha^B_{\pm(2k+1)} \bigg) .
\end{equation}
Their Hirota bilinear equation is given by \cite[equation~(2.1)]{Tak06}
\begin{equation}
	\oint \frac{dz}{2 \pi iz} e^{\xi (t' - t, z)} \tau ( t' - 2[z^{-1}], s') \tau (t + 2[z^{-1}], s )
	=
	\oint \frac{dz}{2 \pi iz} e^{\xi (s' - s,z)} \tau (t', s' - 2[z^{-1}] ) \tau ( t, s + 2[z^{-1}] ) ,
\end{equation}
where
\begin{equation}
	\xi (t, z) = \sum_{k = 0}^\infty t_{2k+1} z^{2k+1} , \quad [z^{-1}] = \Big( z^{-1}, \frac{z^{-3}}{3}, \frac{z^{-5}}{5}, \dotsc \Big) .
\end{equation}
Schur functions are tau-functions for KP, and 2-KP tau-functions can be written as bilinear combinations of Schur functions. For 2-BKP, this role is taken by Schur $Q$-functions, see \cite[proposition~1]{Orl03} for the hypergeometric case. Moreover, using that the algebra $ \mf{b}_\infty$ can be embedded in $ \mf{a}_\infty$ in two independent ways $ \iota_1$, $ \iota_2$, for each $ M \in \widehat{\mathrm{B}}_\infty$, one finds that \cite[equation~(6.7)]{JM83} (indicating by a superscript whether the $\tau$ function is of $KP$ or $BKP$)
\begin{equation}
	\tau^B_M (t_1, t_3, \dotsc )^2 = \tau^A_{\iota_1 (M) + \iota_2 (M)} (t_1, 0, t_3, 0, t_5, \dotsc ).
\end{equation}
This can be extended to 2-(B)KP as in the proof of \cite[theorem~1.1]{Lee20}.

One essential difference between 2-KP and 2-BKP is the grading or charge. The charge of the fermions in the A case allows for an extra lattice ($\Z$-valued) parameter $N$, which promotes 2-KP tau-functions to 2D Toda lattice $\tau_N$ functions. In the neutral, B-type, case, there is only a $ \Z/2\Z$ grading, and moreover $ \tau_0 = \tau_1$ \cite[page~972]{JM83}. Therefore, there is no natural extension from 2-BKP to some kind of ``2D B-Toda lattice".

\section{Localisation formula for spin Gromov--Witten invariants}
\label{sec:localisation}

The goal of this section is to apply the virtual localisation formula of Graber--Pandharipande \cite{GP99} to express spin GW invariants of $\P^1$ in terms of integrals over the fixed points of the torus action, that is integrals over the moduli space of stable curves. The resulting formula will express spin GW invariants of $\P^1$ in terms of double Hodge integrals arising as vertex terms.

\subsection{The torus action}
\label{subsec:torus:action}

We write $V = \C^2$. We consider the torus action on $V$ with weights $(0,1)$, i.e.
\begin{equation}
\begin{aligned}
	\C^{\ast} \times V & \longrightarrow V \\
	\bigl( \xi, (x_1, x_2) \bigr) & \longmapsto \bigl( x_1, \xi\cdot x_2 \bigr) .
\end{aligned}
\end{equation}
We denote by $\P^1$ the projectivisation of $V$. The torus action on $V$ induces an action on $\P^1$ and we denote by $0$ and $\infty$ the fixed points of this action. A lifting of the $\C^{\ast}$ action on $\P^1$ to the line bundle $\mc{O}(-k)$ is determined by the weights $[\alpha,\alpha+k]$ of the representations at the fibres $\mc{O}(-k)_{|0}$ and $\mc{O}(-k)_{|\infty}$ for some integer $\alpha$.

We write $H_{\C^{\ast}}^{*}(\pt) = \Q[t]$ for the equivariant cohomology ring of the point, where $t$ is the first Chern class of the standard representation. The equivariant cohomology of $\P^1$ is then the module over $\Q[t]$ given by:
\begin{equation}
	H_{\C^{\ast}}^{*}(\P^1)= \Q[h,t]/(h^2+th) ,
\end{equation}
where $h$ denotes the equivariant first Chern class of $\mc{O}(1)$. Then the equivariant Poincar\'e dual of $0$ and $\infty$ are given in this basis by
\begin{equation}
	\mathbf{0} = t\cdot 1 + h ,
	\qquad
	\boldsymbol{\infty} = h .
\end{equation}
Let $g,n,m,$ and $d$ be non-negative integers. The torus action on $\P^1$ provides a canonical torus action on the moduli space of stable maps. Thus, we will work in the equivariant cohomology of $\M_{g,n+m}(\P^1,d)$.

In the following, we will denote the generating series of the connected enumerative invariants of interest with a superscript $\circ$.

\begin{definition}\label{def:nPlusmPointGW}
	Define the connected, degree $d$, equivariant $(n+m)$-point function for spin GW invariants of genus $g$ of $(\P^1,\mc{O}(-1))$ as:
	\begin{equation}
		\sGWconn_{g,d}(z,w)
		\coloneqq
		(-1)^{g-1+d} 2^{3g-3+n+m+2d}
		\int_{[\M_{g,n+m}(\P^1, d)]^{\loc,\mc{O}(-1)}}
			\prod_{i =1}^n \frac{z_i \ev_i^{\ast}(\mathbf{0})}{1 - z_i \psi_i}
			\prod_{j =1}^m \frac{w_j \ev_{n+j}^{\ast}(\boldsymbol{\infty})}{1 - w_j \psi_{n+j}} .
	\end{equation}
	The unstable terms are defined as:
	\begin{equation}\label{eq:GWconn:unstable}
	\begin{gathered}
		\sGWconn_{0,0}() = \sGWconn_{1,0}() = 0 \\
		\sGWconn_{0,0}(z_1) = \frac{1}{2t z_1} ,
		\qquad
		\sGWconn_{0,0}(w_1) = - \frac{1}{2t w_1} \\
		\sGWconn_{0,0}(z_1,z_2) = \frac{z_1 z_2}{z_1 + z_2} ,
		\qquad
		\sGWconn_{0,0}(z_1,w_1) = 0 ,
		\qquad
		\sGWconn_{0,0}(w_1,w_2) = \frac{w_1 w_2}{w_1 + w_2} .
	\end{gathered}
	\end{equation}
	Define the corresponding genus generating series and its disconnected counterpart as:
	\begin{equation}\label{eqn:GWconngen}
	\begin{aligned}
		\sGWconn_d(z,w;u)
		&\coloneqq
		\sum_{g=0}^{\infty} u^{g-1} \sGWconn_{g,d}(z,w) , \\
		\sGWdisc_d(z,w;u)
		&\coloneqq
		\sum_{P} \frac{1}{|\Aut(P)|} \prod_{k=1}^{\ell(P)} \sGWconn_{d_k}\bigl( z_{I^{(k)}},w_{J^{(k)}};u \bigr),
	\end{aligned}
	\end{equation}
	where the $P$-sum is over all vectors $(I^{(k)}, J^{(k)}, d_k)_{k=1}^{\ell}$ such that $\bigsqcup_{k=1}^{\ell} I^{(k)} = \{ 1, \dotsc, n\}$, $\bigsqcup_{k=1}^{\ell} J^{(k)} = \{ 1, \dotsc, m\}$, and $\sum_{k=1}^{\ell} {d_k} = d$.
\end{definition}

The restriction of the function $\sGWdisc$ at $t=0$ recovers the generating series of spin GW invariants of $\P^1$:
\begin{equation}
	\sGWconn_{g,d}(z,w)\big|_{t=0}
	=
	(-1)^{g-1+d} 2^{3g-3+n+m+2d}
	\int_{[\M_{g,n+m}(\P^1, d)]^{\loc,\mc{O}(-1)}}
	\prod_{i = 1}^n \frac{z_i \ev_i^{\ast}([{\pt}])}{1 - z_i \psi_i}
	\prod_{j = 1}^m \frac{w_j \ev_{n+j}^{\ast}([{\pt}])}{1 - w_j \psi_{n+j}} .
\end{equation}

\subsection{Applying the localisation formula}
\label{subsec:apply:localisation}

We describe here the computation of the generating function $\sGWdisc_{d}(z,w;u)$ using the virtual localisation formula of Graber--Pandharipande \cite{GP99}. We recall that the fixed loci of the $\C^{\ast}$-action on $\M_{g,n+m}(\P^1, d)$ are indexed by decorated graphs in $\mc{G}_{g,n,m,d}$, that is the set of tuples
\begin{equation}
	\bigl(
		\Gamma,
		V(\Gamma) = V_0 \sqcup V_{\infty},
		\underline{d}=(d_e)_{e \in E(\Gamma)}
	\bigr) ,
\end{equation}
where:
\begin{itemize}
		\item $\Gamma$ is a pre-stable graph. For each vertex $v$, we denote by $g(v), n(v),$ and $\val(v)$ the genus, number of half-edges, and valency (number of adjacent edges) of $v$.

		\item The genus of $\Gamma$ is $g$, i.e: 
		\begin{equation}
			h_1(\Gamma) + \sum_{v \in V(\Gamma)} g(v) = g ,
		\end{equation}
		where $h_1(\Gamma)$ is the first Betti number of the graph and $g(v)$ are the evaluations of the genus function at every vertex.

		\item The partition of $V(\Gamma)$ makes $\Gamma$ bipartite. Moreover, there are $n$ legs attached to vertices in $V_{0}$, and $m$ legs attached to vertices in $V_{\infty}$.
		
		\item The degrees $d_e \in \Z_{+}$ satisfy:
		\begin{equation}
			\sum_{e \in E(\Gamma)} d_e = d ,
		\end{equation}
		where $E(\Gamma)$ is the set of edges of $\Gamma$.
\end{itemize}
The fixed locus $X_{\Gamma,\underline{d}}$ associated to this datum is then isomorphic to 
\begin{equation}
	X_{\Gamma,\underline{d}} \cong \prod_{v\in V(\Gamma)} \M_{g(v),n(v)}
	\lhook\joinrel\xrightarrow{\iota_{(\Gamma,\underline{d})}}
	\M_{g,n}(\P^1,d) ,
\end{equation}
with the convention that if $g(v)=0$ and $n(v)=1$, or $2$ (unstable vertices) then $\M_{g(v),n(v)}$ is a point. By~\cite{GP99} we have:
\begin{equation}\label{eqn:MgnP1vir}
	\bigl[ \M_{g,n}(\P^1,d) \bigr]^\vir
	=
	\sum_{\Gamma,\underline{d}} \iota_{(\Gamma,\underline{d}),\ast} \beta(\Gamma,\underline{d})
	\in A_{*}\bigl( \M_{g,n}(\P^1,d) \bigr)[t,t^{-1}] ,
\end{equation}
where the (Poincaré dual of the) class $ \beta(\Gamma,\underline{d}) $ is given by
\begin{equation}\label{eqn:betagraph}
\begin{split}
	\beta(\Gamma,\underline{d}) = \frac{1}{|\Aut(\Gamma,\underline{d})| \prod_{e\in E(\Gamma)} d_e }
	&
		\prod_{v\in V_0}  t^{g(v)-1} \Lambda \Bigl( -\frac{1}{t} \Bigr) \left(\prod_{e\mapsto v} t^{-d_e} \frac{d_e^{d_e}}{d_e!} \frac{t d_e}{t-d_e\psi_e}\right) \\ 
	\times &
		\prod_{v\in V_\infty} (-t)^{g(v)-1} \Lambda \Bigl( \frac{1}{t} \Bigr) \left(\prod_{e\mapsto v} (-t)^{-d_e} \frac{d_e^{d_e}}{d_e!} \frac{t d_e}{t+d_e\psi_e}\right).
\end{split}
\end{equation}
The convention for the contribution of unstable vertices will be given further. Moreover, along a fixed locus $(\Gamma,\underline{d})$, the different cohomology classes restrict as follows.
\begin{itemize}
	\item
	For all $i\leq n$, the class $\ev_i^{\ast}(\mathbf{0})=t$ if the leg indexed by $i$ is adjacent to a vertex in $V_0$, and vanishes otherwise. Conversely, if $j\geq n+1$, then $\ev_j^{\ast}(\boldsymbol{\infty})=-t$ if the leg indexed by $j$ is adjacent to a vertex in $V_{\infty}$, and vanishes otherwise.

	\item
	Let $k$ be a positive integer, and consider $ R^1\pi_*f^*\mc{O}(-k)$, the first derived pushforward of $\mc{O}(-k)$. It is a vector bundle of rank $g+kd-1$. To localise this class, we may choose a lifting of the torus action to the line bundle $\mc{O}(-k)$. The lifting depends only on the choice of the weights of the $\C^\ast$-action on the fibres of the line bundle at $0$ and $\infty$. As explained before, these weights are of the form $[\alpha,\alpha+k]$ for any choice of integer $\alpha$. Then, given such a lifting, the class $c_{\mathrm{top}}\left(R^1\pi_{\ast}f^{\ast}\mc{O}(-k)\right)$ restricts to a fixed locus of the $\C^\ast$-action as:
	\begin{equation}\label{eqn:general:c:top}
	\begin{split}
		\iota^{*}_{(\Gamma, \underline{d})} c_{\mathrm{top}}\left(R^1\pi_{\ast}f^{\ast}\mc{O}(-k)\right)
		= & 
		\prod_{e\in E(\Gamma)} \prod_{\ell=1}^{kd_e-1} \frac{(\ell + k d_e)(-\alpha t) + \ell (\alpha-k) t }{d_e}\\
		& \times
		\prod_{\mathclap{v\in V_0}} \;\; (k\alpha t)^{\val(v)-1} (\alpha t)^{g(v)} \Lambda \left(-\frac{1}{\alpha t}\right) \\
		& \times
		\prod_{\mathclap{v\in V_\infty}} \;\; (k(\alpha+k) t)^{\val(v)-1} ((\alpha +k)t)^{g(v)} \Lambda \left(-\frac{1}{(\alpha+k) t}\right).
	\end{split}
	\end{equation}
\end{itemize}
A common choice to compute such integrals is $\alpha = -k$ or $0$. Indeed, under such choices, the graphs with vertices of valency greater than $1$ over $\infty$ or $0$ give a trivial contribution. We make a different choice here as it does not provide a symmetric role to both sides of the equation, but it should provide interesting results on the structure of the function $\mathbb{G}$. 

We are now ready to apply the virtual localisation formula to compute the generating function $\sGWdisc_{d}(z,w;u)$.

\begin{theorem}\label{thm:GW:localised:to:H}
	For $d \ge 0$,
	\begin{equation}
		\sGWdisc_{d}(z,w;u)
		=
		\sum_{\mu \in \OP_d}
			\frac{2^{\ell (\mu)}}{\mf{z}_{\mu}} \frac{u^{-d}}{t^n (-t)^{m}}
			\prod_{k = 1}^{\ell(\mu)} \biggl( \frac{u}{t} \frac{(\frac{u\mu_k}{t})^{\mu_k'}}{\mu_k'!} \biggr)
			\dHdisc(\mu, tz; \tfrac{u}{t} ) 
			\prod_{k = 1}^{\ell(\mu)} \biggl( - \frac{u}{t} \frac{(- \frac{u\mu_k}{t})^{\mu_k'}}{\mu_k'!} \biggr)
			\dHdisc(\mu, -tw; -\tfrac{u}{t} ) ,
	\end{equation}
	where $\mf{z}_{\mu} = |\Aut(\mu)| \prod_{i=1}^{\ell(\mu)} \mu_i$ is the order of the stabiliser of any permutation of cycle type $\mu$ under conjugation.
\end{theorem}

\begin{proof}
	Let us start from the connected generating series $\sGWconn_{g,d}(z,w)$. Expressing the localised cycle through \cref{eqn:locsplit} and applying \cref{eqn:MgnP1vir} and the projection formula, we obtain
	\begin{multline*}
		\sGWconn_{g,d}(z,w)
		=
		(-1)^{g-1+d} 2^{3g-3+n+m+2d}
		\sum_{\Gamma,\underline{d}} \int_{X_{\Gamma,\underline{d}}}
			\beta(\Gamma,\underline{d}) \\
			\times
			\prod_{i = 1}^n \frac{z_i \iota^{*}_{(\Gamma, \underline{d})} \ev_i^{\ast}(\mathbf{0})}{1 - z_i \psi_i}
			\prod_{j = 1}^m \frac{w_j \iota^{*}_{(\Gamma, \underline{d})} \ev_{n+j}^{\ast}(\boldsymbol{\infty})}{1 - w_j \psi_{n+j}}
			\iota^{*}_{(\Gamma, \underline{d})} c_{\mathrm{top}}\left( R^1\pi_{\ast}f^{\ast}\mc{O}(-1) \right) .
	\end{multline*}
	For $c_{\mathrm{top}}( R^1\pi_{\ast}f^{\ast}\mc{O}(-1) )$, we use \cref{eqn:general:c:top} with the choices $k = 1$ and $\alpha=-1/2$, which yields a more symmetric formula. Note that this choice of $\alpha$ is not integral, and thus does not correspond to a lifting of the torus action to $\mc{O}(-1)$. However, the class $c_{\mathrm{top}}( R^1\pi_{\ast}f^{\ast}\mc{O}(-1) )$ depends polynomially on $\alpha$ for each fixed locus, thus one may extend the expression of spin GW invariants to all $\alpha \in \C$. For $k = 1$ and $\alpha=-1/2$, \cref{eqn:general:c:top} specialises to
	\begin{equation*}\label{eqn:special:c:top}
		\begin{split}
			\iota^{*}_{(\Gamma, \underline{d})} c_{g-1+d}\left( R^1\pi_{\ast}f^{\ast}\mc{O}(-1) \right)
			= & \; \hphantom{\times}
			\prod_{\mathclap{e\in E(\Gamma)}} \;\; \prod_{\ell=1}^{d_e-1} \frac{( d_e - 2 \ell )t}{2d_e}\\
			& \; \times
			\prod_{\mathclap{v\in V_0}} \;\; \Bigl(-\frac{t}{2}\Bigr)^{g(v) +\val(v)-1} \Lambda \Bigl(\frac{2}{t}\Bigr) \\
			& \; \times
			\prod_{\mathclap{v\in V_\infty}} \;\; \Bigl(\frac{t}{2}\Bigr)^{g(v) +\val(v)-1} \Lambda \Bigl(-\frac{2}{t}\Bigr) .
		\end{split}
	\end{equation*}
	This choice of $\alpha$ imposes that all graphs with at least one edge of even degree do not contribute, as some factor of the form $(d_e - 2\ell)$ vanishes. Combining the above equation with \cref{eqn:betagraph} for $\beta(\Gamma,\underline{d})$ and the pullbacks of the evaluations over $0$ and $\infty$, we obtain 
	\begin{multline}\label{eq:Gconn:Gamma}
		\sGWconn_{g,d}(z,w)
		= (-1)^{g-1+d} 2^{3g-3+n+m+2d}
		\sum_{\Gamma, \underline{d}} \frac{1}{|\Aut(\Gamma, \underline{d})| \prod_e d_e}  \\
		\times
			\prod_{v \in V_0} \biggl(\cont(v) \prod_{e \mapsto v} \cont(e) \biggr)
			\prod_{v \in V_\infty} \biggl(\cont(v) \prod_{e \mapsto v} \cont(e) \biggr) .
	\end{multline}
	The contributions of $(\Gamma,\underline{d})$ are computed as follows.
	\begin{itemize}
		\item
		The contribution of an half-edge $e$ incident to $v\in V_0$ is 
		\begin{equation*}
			\cont(e) = \biggl( \frac{1}{2} \biggr)^{d_e - 1} \frac{d_e^{d'_e}}{d'_e!} t^{-d'_e-1} .
		\end{equation*}
		The contribution of an half-edge $e$ incident to $v \in V_\infty$ is obtained by $t \mapsto -t$.

		\item
		The contribution of a stable vertex $v \in V_0$ is
		\begin{equation*}
			\begin{split}
				\cont(v)
				& =
				\Bigl( -\frac{t}{2} \Bigr)^{\val(v)-1}
				\frac{(-t)^{g(v)}(t/2)^{g(v)}}{t^{3g(v)-2+n(v)}} \\
				& \qquad\qquad \times
					\int_{\M_{g(v),n(v)}}
					\Lambda(2)\Lambda(-1)\prod_{e\mapsto v}
					\frac{d_e}{1 - d_e \psi_e}
					\prod_{i\mapsto v} \frac{tz_i}{1 - tz_i \psi_i} \\
				& =
				\frac{ (-1)^{g(v) - 1 + \val(v)} }{t^{g(v)-1+n(v)-\val(v)} 2^{3g(v)- 3 + n(v) + \val(v)}}
				2^{2g(v) - 2 + n(v)} \\
				& \qquad\qquad \times
				\int_{\M_{g(v),n(v)}}
					\Lambda(2)\Lambda(-1)
					\prod_{e\mapsto v} \frac{d_e}{1 - d_e \psi_e}
					\prod_{i\mapsto v} \frac{tz_i}{1 - t z_i \psi_i} \\
				& =
				\frac{(-1)^{g(v)-1+\val(v)}}{2^{3g(v)-3+n(v)+\val(v)} t^{g(v)-1+n(v)-\val(v)}} 
				\dHconn_{g(v)}(d_e,t z_i) .
			\end{split}
		\end{equation*}
		The contribution of a stable vertex $v \in V_\infty$ is obtained by $t \mapsto -t$.

		\item
		The contributions at unstable vertices are given by the same expressions, if we use our convention of $\dHconn_0(z_1)$ and $\dHconn_0(z_1,z_2)$ given in \cref{eqn:dH:unstable}.
	\end{itemize}
	We can check the above formula for the unstable $1$-point and $2$-point functions, as given in \cref{eq:GWconn:unstable}. The $1$-point functions with a single insertion over $0$ or $\infty$ receive a contribution from a single graph respectively of the form
	\begin{equation*}
		\begin{tikzpicture}[baseline]
			\draw (-1,0) -- (0,0);
			\node at (-1,0) [above] {$z_1$};
			\draw[fill=white] (0,0) circle (4pt);
			\node at (0,0) {\scriptsize$0$};
			\node at (.4,0) {,};

			\begin{scope}[xshift=3cm]
				\draw (1,0) -- (0,0);
				\node at (1,0) [above] {$w_1$};
				\draw[fill=black] (0,0) circle (4pt);
				\node[white] at (0,0) {\scriptsize$0$};
				\node at (1.4,0) {.};
			\end{scope}
		\end{tikzpicture}
	\end{equation*}
	Vertices over $0$ are depicted in white and vertices over $\infty$ are depicted in black. In these cases, the right-hand side of \cref{eq:Gconn:Gamma} gives
	\begin{equation*}
		\sGWconn_{0,0}(z_1) = \dHconn_0(tz_1) = \frac{1}{2t z_1} ,
		\qquad\qquad
		\sGWconn_{0,0}(w_1) = \dHconn_0(-tw_1) = -\frac{1}{2t w_1} ,
	\end{equation*}
	which agree with \cref{eq:GWconn:unstable}. As for the $2$-point functions, there is no graph contributing to the mixed $2$-point function $\sGWconn_{0,0}(z_1,w_1)$ since we are in degree zero. On the other hand, the graphs contributing to the $2$-point functions with two insertions over $0$ or $\infty$ receive a contribution from a single graph respectively of the form
		\begin{equation*}
		\begin{tikzpicture}[baseline]
			\draw (160:1) -- (0,0);
			\node at (160:1) [above] {$z_1$};
			\draw (-160:1) -- (0,0);
			\node at (-160:1) [below] {$z_2$};
			\draw[fill=white] (0,0) circle (4pt);
			\node at (0,0) {\scriptsize$0$};
			\node at (.4,0) {,};

			\begin{scope}[xshift=3cm]
				\draw (20:1) -- (0,0);
				\node at (20:1) [above] {$w_1$};
				\draw (-20:1) -- (0,0);
				\node at (-20:1) [below] {$w_2$};
				\draw[fill=black] (0,0) circle (4pt);
				\node[white] at (0,0) {\scriptsize$0$};
				\node at (1.4,0) {.};
			\end{scope}
		\end{tikzpicture}
	\end{equation*}
	As before, the right-hand side of \cref{eq:Gconn:Gamma} gives
	\begin{align*}
		\sGWconn_{0,0}(z_1,z_2) &= \frac{1}{t} \dHconn_0(tz_1,tz_2) = \frac{z_1 z_2}{z_1 + z_2} , \\
		\sGWconn_{0,0}(w_1,w_2) &= -\frac{1}{t} \dHconn_0(-tw_1,-tw_2) = \frac{w_1 w_2}{w_1 + w_2} .
	\end{align*}
	which agree with \cref{eq:GWconn:unstable}.

	We now take all these contributions together, and make the genus generating series on both sides. For this, we use the following global constraints:
	\begin{align*}
		&
		\sum_{v \in V(\Gamma)} \bigl( g(v) - 1 \bigr) 
		= g - 1 - |E(\Gamma)| ,
		&&
		\sum_{v \in V(\Gamma)} \val(v)
		= 2 |E(\Gamma)| ,
		\\
		&
		\sum_{v \in V(\Gamma)} \bigl( n(v) - \val(v) \bigr)
		= n + m ,
		&&
		\sum_{e \in E(\Gamma)} d_e
		=
		d .
	\end{align*}
	Collecting the powers of $2$ from all half-edges and vertices, we find
	\begin{equation*}
		- \sum_{v \in V(\Gamma)} \bigl(
			3g(v) - 3 + n(v) + \val(v)
		\bigr)
		-2 \sum_{e \in E(\Gamma)} \bigl( d_e - 1 \bigr)
		=
		- (3g - 3 + n + m + 2d) + |E(\Gamma)| .
	\end{equation*}
	In particular, the powers of $2$ cancel out with the global prefactor $2^{3g-3+n+m+2d}$, except for a factor $2^{|E(\Gamma)|}$. Similarly for the powers of $-1$ that are not paired with the equivariant parameter $t$:
	\begin{equation*}
		\sum_{v \in V(\Gamma)} \bigl( g(v) - 1 + \val(v) \bigr)
		=
		g - 1 + |E(\Gamma)|
		\equiv
		g - 1 + d \pmod{2}.
	\end{equation*}
	In the last equation, we used the fact that $|E(\Gamma)| \equiv d \pmod{2}$, since the edge degrees are all odd. In particular, all powers of $-1$ cancel out with the global prefactor $(-1)^{g-1+d}$.

	To conclude, let us analyse the powers of $t$. For each vertex $v$ over $0$ we can pair $-(g(v) - 1)$ powers of $t$ with the genus parameter $u$, leaving out a total of $- \sum_{v \in V_0} ( n(v) - \val(v) ) = - n$. Similarly for vertices over $\infty$, exchanging $t \mapsto -t$. Collecting all together, we have
	\begin{multline*}
			\sGWconn_{d}(z,w; u)
			=
			\sum_{\Gamma, \underline{d}}
				\frac{2^{|E(\Gamma)|}}{|\Aut(\Gamma,\underline{d})| \prod_{e} d_e }
				\frac{u^{|E(\Gamma)|}}{t^n (-t)^m}
			\\
				\prod_{v \in V_0}
				\dHconn \bigl( \{ d_e\}_{e \mapsto v}, \{ tz_i\}_{ i \mapsto v} ; \tfrac{u}{t} \bigr)
				\prod_{e \mapsto v} \frac{1}{t} \frac{( \frac{1}{t} d_e )^{d'_e}}{d'_e!}
			\\
				\prod_{v \in V_\infty}
				\dHconn \bigl( \{d_e\}_{e \mapsto v}, \{ -tw_j\}_{j \mapsto v} ; - \tfrac{u}{t} \bigr)
				\prod_{e \mapsto v} \frac{1}{-t} \frac{( - \frac{1}{t} d_e )^{d'_e}}{d'_e!} .
	\end{multline*}
	We can distribute the factor $u^{|E(\Gamma)|}$ to the half-edges using the relation
	\begin{equation*}
		|E(\Gamma)| = 2 \sum_{e \in E(\Gamma)} \bigl( d_e' + 1 \bigr) - d .
	\end{equation*}
	To change from connected to disconnected counts, we may gather all double Hodge functions over $0$ together, and similarly for $\infty$. Then the sum collapses to graphs with one vertex over $0$ and one over $\infty$, and the only remaining information is the splitting of $d$, which is encoded in an odd partition $\mu$. The factor $|\Aut(\Gamma,\underline{d})| \prod_{e \in E(\Gamma)} d_e$ collapses to $\mf{z}_\mu$, and we simply note that the number of edges coincides with the length of the partition $\mu$.
\end{proof}

\section{Double Hodge integrals in operator formalism}
\label{sec:dH:operator:form}

Employing Fock space techniques in Hurwitz theory dates back to the foundational paper by Okounkov \cite{Oko00}, where the Toda equation for Hurwitz numbers was proved for the first time. Working in Hurwitz theory via Fock space formalism has the significant advantage of providing an explicit representation of operators for computations. In the spirit of \cite[section~2]{OP06b}, this section aims to transfer the operator formalism from the Hurwitz side to the GW side, albeit in the spin setting. The methods for computing the commutation relation needed are analogous to those developed in \cite{OOP20}.

More precisely, the main goal is to express the double Hodge integrals as correlators in the Fock space:
\begin{equation}\label{eqn:dH:corr}
		u^{n} \dHdisc(z_1,\dots,z_n;u) = \Braket{ \prod_{i=1}^n \mc{B} (z_i, uz_i) }
\end{equation}
for some specific operators $\mc{B}$. To this end, we start from the spin ELSV formula \labelcref{eqn:spin:ELSV} which equates the double Hodge integrals with spin Hurwitz numbers for $z_i = \mu_i$ odd positive integers. For these numbers, we know an expression as correlators, specifically as correlators of a product of the $\mc{B}$-operators. We need to extend the definition of the $\mc{B}$-operators from odd integral parameters to complex ones, and show that the associated correlators are rational functions of the parameters. This will allow us to conclude \cref{eqn:dH:corr} via a density argument.

\subsection{The definition of \texorpdfstring{$\mc{B}$}{B}-operators}
\label{subsec:B:op}

\begin{definition}\label{def:B-ops}
	Define the operators
	\begin{equation}
		\mc{B}(z, u z)
		=
		e^{u \frac{z^3}{12}}
		\sum_{k \in \Z} \sum_{j \ge k'+1}
				\frac{\Gamma(z'+1) \Gamma(z + 2j)}{\Gamma(z + k + 1) \Gamma(z' + j + 1)} (u z)^j
				[w^{2j-1}]
					\mc{S}(w)^{z}
					\varsigma(w)^k
					\mc{E}^B_k(w)
		,
	\end{equation}
	where we recall the notation $z' = \frac{z-1}{2}$, the functions $\mc{S}(z) = 2\frac{\sinh (\frac{z}{2})}{z}$, $\varsigma(z) = 2\sinh \left(\frac{z}{2}\right)$, and the B-Okounkov--Pandharipande operators $\mc{E}^B_k$ defined in \labelcref{eqn:B:OP}.
\end{definition}

The main motivation for the above definition is the following result expressing the specialisation of $\mc{B}$ to positive odd integer values of $z$ in terms of bosons and the first completed cut-and-join operator (recall \cref{def:B:operators} for the definition of such operators).

\begin{proposition}\label{prop:B:operators:odd:integers}
	For $\mu \in \Z_+^{\textup{odd}}$, the operator $\mc{B}(\mu,u\mu)$ is the conjugation of $\alpha^B_{-\mu}$ by $e^{\alpha^B_1} e^{u\frac{\mc{F}^B_{3}}{3}}$, up to the non-polynomial part:
	\begin{equation}
		\mc{B}(\mu, u\mu)
		=
		\frac{\mu'!}{(u \mu)^{\mu'}}
		e^{\alpha^B_1} e^{u\frac{\mc{F}^B_{3}}{3}} \alpha^B_{-\mu} e^{-u\frac{\mc{F}^B_{3}}{3}} e^{-\alpha^B_1}.
	\end{equation}
\end{proposition}

The proof is obtained by expanding both sides on the basis $E_{i,j}^B$ from \cref{def:b:inf}. Such expansion for the left hand-side was computed in \cite[proposition 7.8]{GKL21}, and we recall it here for the reader's convenience. For this convention for of the basis elements $E^B_{j,k}$, we refer to \cite[section~12.1]{Gia21}.

\begin{lemma}[{\cite{GKL21}}]
	Let $\Delta$ be the backward difference operator, and denote $Q(l,\mu) = \frac{(l+\mu)^3 - l^3}{3\mu}$. Then
	\begin{multline}\label{eqn:conj:on:basis}
		\frac{\mu'!}{(u \mu)^{\mu'}}
		e^{\alpha^B_1} e^{u\frac{\mc{F}^B_{3}}{3}} \alpha^B_{-\mu} e^{-u\frac{\mc{F}^B_{3}}{3}} e^{-\alpha^B_1} = \\
		=
		\sum_{s \in \Z}
			\frac{(u \mu)^s}{(\mu'+1)_s}
			\Biggl(
			\sum_{j \ge 0} \sum_{l \ge \frac{j+1-\mu}{2}}
				(-1)^l
				\frac{\Delta^j}{j!}
				Q(l,\mu)^{\mu'+s}
			E^B_{l+\mu-j,l}
		+
		\frac{1}{2} \frac{\Delta^{\mu-1}}{\mu!}
			Q(l,\mu)^{\mu'+s} \big|_{l=0}
			\Id
		\Biggr) .
	\end{multline}
\end{lemma}

To expand the left-hand side, we use the following lemma.

\begin{lemma}\label{lem:bckwrds:Q}
	The backwards difference operator acts on monomials as:
	\begin{equation}\label{eqn:diff:monomial}
		\frac{\Delta^j}{j!} \frac{l^a}{a!}
		=
		[w^a] \frac{\varsigma(w)^{j}}{j!} e^{w(l-\frac{j}{2})} .
	\end{equation}
	Moreover, the extraction of powers of $l$ from the polynomial $Q$ yields
	\begin{equation}\label{eqn:coeff:Q}
		a! [l^a] Q(l,\mu)^{p}
		=
		[w^{-a}]
		\left(
			\sum_{m=0}^p \binom{p}{m} \frac{(2p-2m)!}{w^{2p-2m}} \frac{\mu^{2m}}{12^m}
		\right)
			e^{w\frac{\mu}{2}}.
	\end{equation}
\end{lemma}

\begin{proof}
	We may expand the left-hand side of \cref{eqn:diff:monomial} as follows:
	\begin{align*}
		\frac{\Delta^j}{j!} \frac{l^a}{a!}
		&= \frac{1}{j!} \sum_{i=0}^j \binom{j}{i}(-1)^i \frac{(l-i)^a}{a!}
		\\
		&= \frac{1}{j!} \sum_{i=0}^j \binom{j}{i} (-1)^i \sum_{k=0}^a \frac{1}{(a-k)!}(j-i)^{a-k} \frac{1}{k!}(l-j)^k
		\\
		&= \frac{1}{j!} \sum_{k=0}^a \sum_{i=0}^j \binom{j}{i} (-1)^i \frac{1}{(a-k)!}(j-i)^{a-k} [w^k]e^{w(l-j)} .
	\end{align*}
	The $i$-sum can be expressed as a Stirling number of the second kind $\Strlng{n}{j}{}$, thanks to the identity $\Strlng{n}{j}{} = \frac{1}{j!} \sum_{i=0}^{t} (-1)^{i} (t-i)^{n} \binom{j}{i}$. Thus:
	\begin{align*}
		\frac{\Delta^j}{j!} \frac{l^a}{a!}
		&= \sum_{k=0}^a \frac{1}{(a-k)!} \Strlng{a-k}{j}{} [w^k]e^{w(l-j)}
		\\
		&= \frac{1}{j!} \sum_{k=0}^a [w^{a-k}] (e^w-1)^j [w^k]e^{w(l-j)}
		\\
		&= [w^a] \frac{\varsigma(w)^{j}}{j!} e^{w(l-\frac{j}{2})}.
	\end{align*}
	The second equality follows from the identity $\frac{1}{n!} \Strlng{n}{j}{} = \frac{1}{j!} [w^n] (e^w - 1)^j$. For \cref{eqn:coeff:Q}, we find
	\begin{equation*}
	\begin{split}
		a! [l^a] Q(l,\mu)^{p}
		&= a! [l^a] \Big( l^2 + \mu l + \frac{\mu^2}{3} \Big)^p
		\\
		&= a! [l^a] \Big( \big(l + \frac{\mu}{2}\big)^2 + \frac{\mu^2}{12} \Big)^p
		\\
		&= \sum_{m=0}^p \binom{p}{m} \frac{\mu^2m}{12^m} a! [l^a]\Big( l + \frac{\mu}{2} \Big)^{2p-2m}
		\\
		&= \sum_{m=0}^p \binom{p}{m} \frac{(2p-2m)!}{(2p-2m-a)!} \frac{\mu^{2m}}{12^m} \Big( \frac{\mu}{2} \Big)^{2p-2m-a}
		\\
		&=
		[w^{-a}]
		\Big(
			\sum_{m=0}^p \binom{p}{m} \frac{(2p-2m)!}{w^{2p-2m}} \frac{\mu^{2m}}{12^m}
		\Big)
			e^{w\frac{\mu}{2}} . \hfill \qedhere
	\end{split}
	\end{equation*}
\end{proof}

\begin{proof}[Proof of \cref{prop:B:operators:odd:integers}]
	Starting from the left-hand side of \cref{eqn:conj:on:basis}, we find
	\begin{multline*}
		\frac{\mu'!}{(u \mu)^{\mu'}}
		e^{\alpha^B_1} e^{u\frac{\mc{F}^B_{3}}{3}} \alpha^B_{-\mu} e^{-u\frac{\mc{F}^B_{3}}{3}} e^{-\alpha^B_1} = \\
		=
		\sum_{s \ge -\mu'}
			\frac{(u \mu)^s}{(\mu'+1)_s}
			\Bigg(
			\sum_{k = -\mu}^{2s-1}
			\sum_{l \ge k'+1}
				(-1)^l
				\frac{\Delta^{\mu+k}}{(\mu+k)!}
				Q(l,\mu)^{\mu'+s}
			E^B_{l-k,l}
			+
			\frac{1}{2} \frac{\Delta^{\mu-1}}{\mu!}
				Q(l,\mu)^{\mu'+s} \big|_{l=0}
				\Id
		\Bigg) .
	\end{multline*}
	We can now expand the difference operator using \cref{lem:bckwrds:Q} to get
	\[
		\frac{\Delta^{\mu+k}}{(\mu+k)!} Q(l,\mu)^{\mu'+s}
		=
		\sum_{m=0}^{\mu'+s}
				\binom{\mu'+s}{m} \frac{(\mu+2(s-m)-1)!}{(\mu+k)!} \frac{\mu^{2m}}{12^m}
				[w^{2(s-m)-1}]
					\mc{S}(w)^\mu \varsigma(w)^k e^{w(l - \frac{k}{2})} .
	\]
	Similarly for the contribution of the identity:
	\[
		\frac{1}{2} \frac{\Delta^{\mu-1}}{\mu!} Q(l,\mu)^{\mu'+s} \big|_{l=0}
		=
		\frac{1}{2} \sum_{m=0}^{\mu'+s}
				\binom{\mu'+s}{m} \frac{(\mu+2(s-m)-1)!}{\mu!} \frac{\mu^{2m}}{12^m}
				[w^{2(s-m)-1}]
					\mc{S}(w)^{\mu} \frac{e^{\frac{w}{2}}}{\varsigma(w)} .
	\]
	Notice that we are extracting the coefficient of an odd power of $w$ from $\mc{S}(w)^{\mu} \varsigma(w)^{-1} e^{\frac{w}{2}}$, and $\mc{S}(w)^{\mu} \varsigma(w)^{-1}$ is an odd function. In particular, we can substitute $e^{\frac{w}{2}}$ with its even part, $\cosh(w/2)$, since it would not change the extracted coefficient. Thus, we find
	\[
		\frac{1}{2} \frac{\Delta^{\mu-1}}{\mu!} Q(l,\mu)^{\mu'+s} \big|_{l=0} = \\
		=
		\frac{1}{2} \sum_{m=0}^{\mu'+s}
				\binom{\mu'+s}{m} \frac{(\mu+2(s-m)-1)!}{\mu!} \frac{\mu^{2m}}{12^m}
				[w^{2(s-m)-1}]
					\mc{S}(w)^{\mu} \frac{\cosh(\tfrac{w}{2})}{\varsigma (w)} .
	\]
	In particular, we get the right correction for $\mc{E}_0^B(w)$. All together, exchanging the $s$- and $k$-sums, we find 
	\begin{multline*}
		\frac{\mu'!}{(u \mu)^{\mu'}}
		e^{\alpha^B_1} e^{u\frac{\mc{F}^B_{3}}{3}} \alpha^B_{-\mu} e^{-u\frac{\mc{F}^B_{3}}{3}} e^{-\alpha^B_1}
		=
		\sum_{k \ge -\mu} \sum_{s \ge k'+1}
			\frac{(u \mu)^s}{(\mu'+1)_s} \\
			\times
			\sum_{m=0}^{\mu'+s}
				\binom{\mu'+s}{m} \frac{(\mu+2(s-m)-1)!}{(\mu+k)!}
				\frac{\mu^{2m}}{12^m}
				[w^{2(s-m)-1}]
					\mc{S}(w)^\mu \varsigma(w)^k
					\mc{E}^B_k(w) .
	\end{multline*}
	Notice that the factorials simplify as
	\[
		\frac{1}{(\mu'+1)_s}\binom{\mu'+s}{m} \frac{(\mu+2(s-m)-1)!}{(\mu+k)!}
		=
		\frac{\mu'!}{(\mu+k)!} \frac{1}{m!} \frac{(\mu+2(s-m)-1)!}{(\mu'+s-m)!} .
	\]
	Exchanging the $s$- and $m$-sums, we should get bounds $m \ge 0$ and $s \ge \max\set{m-\mu',k'+1}$. However, the summands for $s < m+k'+1- \frac{\delta_{k,0}}{2}$ vanish, due to the vanishing order of zeros of $\varsigma(w)^k \mc{E}^B_k(w)$. Hence, for $k \ne 0$, we can replace the lower bound for $s$ by $m+k'+1$ in order to get
	\[
	\begin{split}
		& \frac{\mu'!}{(u \mu)^{\mu'}}
		e^{\alpha^B_1} e^{u\frac{\mc{F}^B_{3}}{3}} \alpha^B_{-\mu} e^{-u\frac{\mc{F}^B_{3}}{3}} e^{-\alpha^B_1} = \\
		& \qquad\quad =
		\sum_{k \ge -\mu} \frac{\mu'!}{(\mu+k)!} 
		\sum_{m \ge 0}
			\frac{1}{m!} \frac{\mu^{2m}}{12^m}
			\sum_{s \ge m + k'+1}
				\frac{(\mu+2(s-m)-1)!}{(\mu'+s-m)!} (u \mu)^s
				[w^{2(s-m)-1}]
					\mc{S}(w)^\mu \varsigma(w)^k
					\mc{E}^B_k(w)
		\\
		& \qquad\quad =
		e^{u \frac{\mu^3}{12}}
		\sum_{k \ge -\mu} \sum_{j \ge k'+1}
			\frac{\mu'!}{(\mu+k)!} 
			\frac{(\mu+2j-1)!}{(\mu'+j)!} (u \mu)^t
			[w^{2j-1}]
				\mc{S}(w)^\mu
				\varsigma(w)^k
				\mc{E}^B_k(w)
		.
	\end{split}
	\]
	We can now see that the above expression coincides with the definition of $\mc{B}$ for $z = \mu \in \Z^{\textup{odd}}_+$. Indeed, we have the factor $\frac{\Gamma(\mu' + 1)}{\Gamma(\mu + k + 1)}$, which vanishes for $k < - \mu$ (and has no poles, since $\mu$ is assumed to be positive). On the other hand, by Legendre duplication formula, we find
	\[
		\frac{\Gamma(\mu + 2j)}{\Gamma(\mu' + j + 1)} = \frac{2^{\mu + 2j -1}}{\sqrt{\pi}} \Gamma(\mu' + j + \tfrac{1}{2}),
	\]
	which has no poles for $\mu \in \Z^{\textup{odd}}_+$ and $j$ integer. Thus, the summation in the definition of $\mc{B}$ for $z = \mu \in \Z^{\textup{odd}}_+$ reduces to $k \ge -\mu$.
\end{proof}

\subsection{The commutation relation}
\label{subsec:comm:rel}

We can now compute the commutator of two $\mc{B}$-operators. Consider the doubly infinite series
\begin{equation}
	\delta(z,-w) = \frac{1}{w} \sum_{k\in \Z} \Big( -\frac{z}{w}\Big )^k \in \Q\pparenthesis{z,w} ,
\end{equation}
which is a formal $\delta$-function at $z + w = 0$ in the sense that $(z+w)\delta(z,-w) = 0$, or equivalently as the discontinuity of the rational function $\frac{1}{(z+w)}$ along the divisor $z + w = 0$.
The same series appears in the A-case \cite[theorem~1]{OP06b}.

\begin{proposition}[Commutation relation] \label{prop:commutation}
	\begin{equation}
		\bigl[
			\mc{B}(z,uz),\mc{B}(w,uw)
		\bigr]
		= 
		u zw \delta(z,-w) .
	\end{equation}
\end{proposition}

Before giving a proof of the above commutation relation, let us explore its main consequences. Firstly, the correlator $\Braket{ \mc{B}(z_1, uz_1) \cdots \mc{B}(z_n, uz_n)}$ is a rational function with at most simple poles at $z_i = 0$ and simple poles along the antidiagonals $z_i + z_j = 0$.

\begin{corollary}
	We have
	\begin{equation}
		\prod_{i=1}^n z_i \prod_{1 \leq i < j \leq n}(z_i + z_j)
		\Braket{ \mc{B}(z_1, uz_1) \cdots \mc{B}(z_n, uz_n)}
		\in \Q[u^{\pm}] \bbraket{z_1, \dots, z_n}.
	\end{equation}
\end{corollary}

\begin{proof}
	By \cref{prop:commutation}, the correlator, once multiplied with the product of $ z_i + z_j$, must be symmetric in the $z_i$. It is also a Laurent series in $z_1$ (and hence the other $z_i$) with at most a simple pole, as can be seen immediately from \cref{def:B-ops}. 
\end{proof}

Thanks to this regularity result, we can prove the main result of the section.

\begin{theorem}\label{thm:dH:as:VEV}
	The generating series of double Hodge integrals can be expressed as the following vacuum expectation value:
	\begin{equation}
		u^{n} \dHdisc(z_1,\dots,z_n;u) = \Braket{ \prod_{i=1}^n \mc{B} (z_i, uz_i) } .
	\end{equation}
\end{theorem}

\begin{proof}
	Recall the spin ELSV formula (\cref{thm:spin:ELSV}) and the expression of spin Hurwitz numbers as vacuum expectation values (\cref{prop:sHN:VEV}):
	\begin{align*}
		h^{\bullet, +, 2}_{g;\mu}
		& =
		2^{g-1+n}
		\bigg( \prod_{i=1}^n \frac{\mu_i^{\mu_i'}}{\mu_i'!} \bigg)
		\int_{\M^\bullet_{g,n}} \frac{ \Lambda (2) \Lambda (-1)}{\prod_{i=1}^n (1-\mu_i \psi_i)}
		,
		\\
		h^{\bullet, +, 2}_{g;\mu}
		& =
		2^{1-g} [u^b]
		\Braket{
			\prod_{i=1}^n
			e^{\alpha^B_1} e^{u\frac{\mc{F}^B_3}{3}}
			\frac{\alpha^B_{-\mu_i}}{\mu_i}
			e^{-u \frac{\mc{F}^B_3}{3}} e^{-\alpha^B_1}
		} .
	\end{align*}
	Here 
	$b = g-1 + \frac{|\mu| +n}{2}$ is given by the Riemann--Hurwitz formula. On the other hand, from \cref{prop:B:operators:odd:integers} we find
	\begin{equation*}
		\frac{(u\mu)^{\mu'}}{\mu'!} \mc{B}(\mu, u\mu)
		=
		e^{\alpha^B_1} e^{u\frac{\mc{F}^B_3}{3}} \alpha^B_{-\mu_i} e^{-u \frac{\mc{F}^B_3}{3}} e^{-\alpha^B_1} .
	\end{equation*}
	Putting this together, we obtain:
	\begin{align*}
		\dHdisc (\mu;u) 
		& =
		\sum_{g=0}^\infty u^{g-1} 2^{2g-2+n} \int_{\M^\bullet_{g,n}} \Lambda(2) \Lambda(-1) \prod_{i=1}^n \frac{\mu_i}{1-\mu_i \psi_i}
		\\
		& =
		\sum_{g=0}^\infty u^{g-1} \biggl( \prod_{i=1}^n \mu_i \biggr)
			\biggl( \prod_{i=1}^n \frac{\mu_i^{\mu_i'}}{\mu_i'!} \biggr)^{-1}
			\bigl[ v^{g-1+\frac{|\mu| + n}{2}} \bigr]
			\Braket{ \prod_{i=1}^n e^{\alpha^B_1} e^{v\frac{\mc{F}^B_3}{3}} \frac{\alpha^B_{-\mu_i}}{\mu_i} e^{-v \frac{\mc{F}^B_3}{3}} e^{-\alpha^B_1} }
		\\
		& =
		\sum_{g=0}^\infty u^{g-1} [v^{g-1 +n}] \Braket{ \prod_{i=1}^n \mc{B} (\mu_i, v\mu_i) }
		\\
		& =
		u^{-n} \Braket{ \prod_{i=1}^n \mc{B} (\mu_i, u\mu_i) } .
	\end{align*}
	Since both $\dHdisc (\mu;u)$ and $\Braket{ \prod_{i=1}^n \mc{B} (\mu_i, u\mu_i) }$ are rational functions and they agree on a Zariski dense set, they coincide for all $z \in \C^n$.
\end{proof}

The rest of the section is dedicated to the proof of the commutation relation between $\mc{B}$-operators (\cref{prop:commutation}). We will pursue a similar approach to the A-case in \cite{OOP20}. Recall from that paper the algebra
\begin{equation}
	\mf{h} = \Z \braket{ S^\pm, H^\pm } /([H,S] - S) .
\end{equation}
We can see $\mf{h}$ as a subalgebra of the bi-infinite general linear algebra $\mf{a}_{\infty}$, with the identification $H = \mc{F}_1$ and $S = \alpha_{-1}$. Let $\widehat{\mf{h}}$ be its central extension induced from the one of $\widehat{\mf{a}}_{\infty}$. In fact $\widehat{\mf{h}}$ can be described explicitly as the Lie subalgebra spanned by the Okounkov--Pandharipande operators $ \mc{E}_{m,k} \coloneqq [z^k] \mc{E}_m (z)$.

Let us now define the Lie algebra $\widehat{\mf{h}}^B$ to be the subalgebra of $\hat{\mf{b}}_\infty$ spanned by $ \mc{E}_{m,k}^B \coloneqq [z^k] \mc{E}^B_m (z)$.

\begin{proposition}~
	\begin{enumerate}
		\item If $ m + k $ is even, then $ \mc{E}^B_{m,k} = 0$.

		\item The map $ p \colon \widehat{\mf{h}} \to \widehat{\mf{h}}^B \colon \mc{E}_{m,k} \mapsto \mc{E}^B_{m,k} $ is a surjective Lie algebra morphism, with right inverse $ \iota \colon \widehat{\mf{h}}^B \to \widehat{\mf{h}} \colon \mc{E}^B_{m,k} \mapsto \mc{E}_{m,k}$ for odd $ m+k$.
	\end{enumerate}
\end{proposition}

\begin{proof}
	The first point follows immediately from the parity relations of \cref{prop:OP:VEV}. For the second point, it is clear that $ p \circ \iota = \Id_{\widehat{\mf{h}}^B}$, and this implies surjectivity of $ p$. To show these are Lie algebra morphisms, we compare commutators. From \cite[equation~(2.17)]{OP06a} and the commutation relations of \cref{prop:OP:VEV}, we obtain:
		\begin{align*}
			\bigl[ \mc{E}_m(z),\mc{E}_n(w) \bigr]
			&=
			\varsigma \big( \det \begin{bsmallmatrix} m & z \\ n & w \end{bsmallmatrix} \big) \mc{E}_{m+n}(z+w);
			\\
			\bigl[ \mc{E}^B_m(z),\mc{E}^B_n(w) \bigr]
			&=
			\frac{1}{2} \varsigma \big( \det \begin{bsmallmatrix} m & z \\ n & w \end{bsmallmatrix} \big) \mc{E}^B_{m+n}(z+w)
			+
			\frac{(-1)^{n+1}}{2}\varsigma \big( \det \begin{bsmallmatrix} m & z \\ n & -w \end{bsmallmatrix} \big) \mc{E}^B_{m+n}(z-w)
			\\
			&=
			\frac{1}{2} \varsigma \big( \det \begin{bsmallmatrix} m & z \\ n & w \end{bsmallmatrix} \big) \mc{E}^B_{m+n}(z+w)
			+
			\frac{(-1)^{m+1}}{2}\varsigma \big( \det \begin{bsmallmatrix} m & -z \\ n & w \end{bsmallmatrix} \big) \mc{E}^B_{m+n}(w-z).
		\end{align*}
		Because the second is the symmetrisation or antisymmetrisation of the first in $w$, depending on the parity of $n$ (or in $z$, depending on the parity of $m$, or both) we find that if both $m+ k$ and $n + l$ are odd, then
		\begin{equation*}
			\bigl[ p(\mc{E}_{m,k}) , p(\mc{E}_{n,l}) \bigr]
			=
			p \big( [\mc{E}_{m,k}, \mc{E}_{n,l}]\big) ,
		\end{equation*}
		because the (anti)symmetrisation preserves this coefficient. If either or both of $ m+k$ and $n+l$ are even, the left-hand side of the above equation vanishes by definition of $ \mc{E}^B_{m,k}$ or of $\mc{E}^B_{n,l}$, and the right-hand side vanishes because the (anti)symmetrisation kills that coefficient. Similarly for $\iota$. 
\end{proof}

By this proposition we can perform all of our calculations in $\mf{h}$, then push them forward to $ \widehat{\mf{h}}$ and eventually to $ \widehat{\mf{h}}^B$ along $p$. 
The computation of the commutator between $\mathcal{B}(z,uz)$ operators is achieved by computing the commutator between their leading order in the $u$-variable, and showing that there exists a dressing operator $W$ relating $\mathcal{B}(z,uz)$ and their leading order in $u$ by conjugation. The following lemma defines the dressing operator $W$.

\begin{lemma}\label{def:dressing}
	Consider the operators in $\mf{h}$ given by
	\begin{equation}
		D = S^{-1} + H,
		\qquad\quad
		\tilde{D} = S^{-1} + X, 
		\qquad\quad
		E =
		\frac{1}{3} \big(
			X^3 + S^{-1} X^2 + X S^{-1} X + X^2 S^{-1}
		\big) .
	\end{equation}
	where $X = -\frac{1}{2} \{ H, \frac{Z}{1-Z} \}$ and $Z = \frac{t}{2u} S^2$. Then
	\begin{equation}\label{eq:Dtilde:ODE}
		\frac{d \tilde{D}}{du} + \frac{1}{t} \bigl[ E,\tilde{D} \bigr] = 0,
	\end{equation}
	and the unique solution $W$ of the linear differential equation
	\begin{equation}\label{eq:W:ODE}
		\frac{dW}{du} = \frac{1}{t} WE
	\end{equation}
	such that $W|_{u=0} = 1$ is upper unitriangular and satisfies $W^{-1} D W = \tilde{D}$.
\end{lemma}

\begin{proof}
	We start with the proof of \cref{eq:Dtilde:ODE}. We simply have $\frac{d \tilde{D}}{du} = \frac{1}{2u} \{ H, \frac{Z}{(1-Z)^2} \}$. On the other hand, direct computation shows that
	\begin{equation*}
		\bigl[ E,\tilde{D} \bigr]
		=
		\frac{1}{3} \Bigl(
			2[X^2,S^{-2}] + X[S^{-1},X]S^{-1} - S^{-1}[X,S^{-1}]X
		\Bigr).
	\end{equation*}
	The commutator between $X$ and $S^{-1}$ can be simplified, using the commutation relation $[H,S^{-1}] = -S^{-1}$ and the fact that $S^{-1}$ commutes with any expression in $Z$:
	\begin{equation*}
	\begin{split}
		[X,S^{-1}]
		=
		- \frac{1}{2} \bigg[ \Big\{ H , \frac{Z}{1-Z} \Big\}, S^{-1} \bigg]
		=
		- \frac{1}{2} \Big\{ \big[ H , S^{-1} \big] , \frac{Z}{1-Z} \Big\}
		=
		S^{-1} \frac{Z}{1-Z}.
	\end{split}
	\end{equation*}
	Thus, we find $[E,\tilde{D}] = \frac{1}{3} ( 2[X^2,S^{-2}] - \{ X, S^{-2}\frac{Z}{1-Z} \} ) = \frac{1}{3} ( 2[X^2,S^{-2}] - \frac{t}{2u} \{ X, \frac{1}{1-Z} \} )$. A lengthy but straightforward computation shows that
	\begin{enumerate}
		\item $[ X^2 , S^{-2} ] = \frac{t}{u} \{ X , \frac{1}{1-Z} \}$
		\item $\{ X , \frac{1}{1-Z} \} = - \{ H, \frac{Z}{(1-Z)^2} \}$
	\end{enumerate}
	Thus, we conclude as
	\begin{equation*}
		\frac{1}{t} [E,\tilde{D}]
		=
		\frac{1}{3t} \bigg(
			-\frac{2t}{u} \Big\{ H, \frac{Z}{(1-Z)^2} \Big\}
			+ \frac{t}{2u} \Big\{ H, \frac{Z}{(1-Z)^2} \Big\}
		\bigg)
		=
		- \frac{1}{2u} \Big\{ H, \frac{Z}{(1-Z)^2} \Big\}
		=
		- \frac{d\tilde{D}}{du},
	\end{equation*}
	hence proving \cref{eq:Dtilde:ODE}. As \cref{eq:W:ODE} has no singularity at $ u=0$, there does indeed exist a unique solution such that $ W|_{u=0} = 1$. This solution is upper unitriangular because $E$ is strictly upper-triangular. On the other hand, by direct computation,
	\begin{equation*}
		\frac{d}{du} \big( W \tilde{D} W^{-1} D^{-1} \big)
		=
		W \Bigl( \frac{1}{t} [E, \tilde{D}] + \frac{d\tilde{D}}{du} \Bigr) W^{-1} D^{-1}
		=
		0 .
	\end{equation*}
	Together with $ \tilde{D}|_{u=0} = S^{-1} + H = D$ and $ W|_{u=0} = 1$, we get $ W^{-1} D W = \tilde{D}$ for all $ u$.
\end{proof}

Define the operators
\begin{equation}\label{eqn:straightB}
	\ms{B}(z)
	\coloneqq \frac{1}{u} \mc{B} (tz,uz) ,
	\qquad\qquad
	\widetilde{\ms{B}} (z)
	\coloneqq 
	\frac{1}{u} \sum_{l \in \Z} \frac{(uz)^l}{(tz'+1)_l} \alpha^B_{2l-1} .
\end{equation}
The latter is obtained from the former by selecting the leading term in the $u$-variable for each energy level of the operator. Both operators are clearly homogeneous of degree $-1$ with respect to the degree given by
\begin{equation}\label{eq:deg:utz}
	\deg u = \deg t = -\deg z = 1.
\end{equation}

\begin{theorem}\label{thm:Bdressing}
	Let $ W$ be as in \cref{def:dressing}. Then 
	\begin{equation}\label{DressingOfB}
		W^{-1} \ms{B} W = \widetilde{\ms{B}}.
	\end{equation}
\end{theorem}

\begin{proof}
	As explained above, and in \cite{OOP20}, both $\ms{B}$ and $W$ are sums of monomials $H^a S^b$ with coefficients in the ring of Laurent polynomials of variables $u, t, z$ and, more-over, are homogeneous with respect to grading \cref{eq:deg:utz}. Therefore, using Zariski density, we need only prove \cref{DressingOfB} at the values $t = 1$ and $z = m \in \Z^{\textup{odd}}_{+}$. We define the operators
	\begin{equation*}
		\ms{B}^{(m)} \coloneqq u \frac{(um)^{m'}}{(m')!} \ms{B}|_{t=1,z=m},
		\qquad
		\widetilde{\ms{B}}^{(m)} \coloneqq u \frac{(um)^{m'}}{(m')!} \widetilde{\ms{B}}|_{t=1,z=m}.
	\end{equation*}
	Then by \cref{prop:B:operators:odd:integers} and direct computation,
	\begin{equation*}
		\ms{B}^{(m)} = e^{1/S} e^{uH^3/3} S^m e^{-uH^3/3} e^{-1/S}, \qquad \widetilde{\ms{B}}^{(m)} = S^m e^{um/S^2}.
	\end{equation*}
	In particular, $ \ms{B}^{(m)} = (\ms{B}^{(1)})^m $ and $ \widetilde{\ms{B}}^{(m)} = (\widetilde{\ms{B}}^{(1)})^m$, so we only need to prove that
	\begin{equation*}
	 O \coloneqq W^{-1} \ms{B}^{(1)} W = \widetilde{\ms{B}}^{(1)} \eqqcolon \tilde{O}.
	\end{equation*}
	This equation holds at $ u = 0$, as $ \ms{B}^{(1)}|_{u=0} = \widetilde{\ms{B}}^{(1)}|_{u=0} = S$ and $ W|_{u=0} = 1$. Taking the $ u $ derivative of $O$, we get
	\begin{equation*}
		\frac{dO}{du} 
		=
		[O, E] + \frac{1}{3}W^{-1} e^{1/S} e^{uH^3/3} [H^3,S] e^{-uH^3/3} e^{-1/S} W 
		=
		\big[O, E - \tfrac{1}{3} \tilde{D}^3\big] .
	\end{equation*}
	By uniqueness of solutions of linear ODEs, it is enough to show that the same equation holds for $\tilde{O}$. The left-hand side is computed as $$\frac{d\tilde{O}}{du} = S^{-1} e^{u/S^2}.$$ On the other hand $E - \frac{1}{3}\tilde{D}^3 = - \frac{1}{3} \big( S^{-3} + S^{-2}X + S^{-1}XS^{-1} + XS^{-2} \big)$. Thus,
	\begin{equation*}
		\Big[ \tilde{O} , E - \frac{1}{3}\tilde{D}^3 \Big]
		=
		- \frac{1}{3} \Big( S^{-2}[\tilde{O},X] + S^{-1}[\tilde{O},X]S^{-1} + [\tilde{O},X]S^{-2} \Big) .
	\end{equation*}
	From the commutation relation $[\tilde{O} , H] = 2u S^{-1}(1 - Z) e^{u/S^2}$, we deduce the commutation relation $[ \tilde{O} , X ] = -\frac{1}{2} \{ [ \tilde{O}, H ], \frac{Z}{1-Z} \} = -S e^{u/S^2}$. Finally we obtain
	\begin{equation*}
		\Big[ \tilde{O} , E - \frac{1}{3}\tilde{D}^3 \Big]
		=
		S^{-1}e^{u/S^2}
		= \frac{d \tilde{O}}{du}. \qedhere
	\end{equation*}
\end{proof}

\begin{lemma}
	As elements of $ \widehat{\mf{h}}^B$,
	\begin{equation}
		\bigl[ \widetilde{\ms{B}}(z),\widetilde{\ms{B}}(w) \bigr] = \frac{zw}{u} \delta(z,-w) .
	\end{equation}
\end{lemma}

\begin{proof}
	We may set $ t=1$ for the proof by homogeneity. By definition,
	\begin{equation*}
	\begin{split}
		 \bigl[ \widetilde{\ms{B}}(z),\widetilde{\ms{B}}(w) \bigr]
		 &=
		\frac{1}{u^2} \bigg[
			\sum_{k \in \Z} \frac{(uz)^k}{(z'+1)_k} \alpha^B_{2k-1},
			\sum_{l \in \Z} \frac{(uw)^l}{(w'+1)_l} \alpha^B_{2l-1}
		\bigg]
		\\
		&= \frac{1}{u^2}
		\sum_{k,l \in \Z} \frac{(uz)^k}{(z'+1)_k} \frac{(uw)^l}{(w'+1)_l} (k-\tfrac{1}{2}) \delta_{k+l-1}
		\\
		&=
		\frac{w}{u}\sum_{k \in \Z}(k-\tfrac{1}{2})\Big( \frac{z}{w}\Big)^k \frac{1}{(z'+1)_k} \frac{1}{(w'+1)_{1-k}} .
	\end{split}
	\end{equation*}
	Let us split this sum in positive $k$ and non-negative $k$: $[\widetilde{\ms{B}}(z),\widetilde{\ms{B}}(w) ] = C_+ + C_-$. Then
	\begin{equation*}
	\begin{split}
		\Big(1+ \frac{z}{w} \Big)C_+ 
		&=
		\Big(1+ \frac{z}{w} \Big) \frac{w}{u}\sum_{k = 1}^\infty(k-\tfrac{1}{2})\Big( \frac{z}{w}\Big)^k \frac{1}{(z'+1)_k (w'+1)_{1-k}} 
		\\
		&=
		\frac{w}{u}\sum_{k = 1}^\infty \bigg( \Big( \frac{z}{w}\Big)^k \frac{k-\tfrac{1}{2}}{(z'+1)_k (w'+1)_{1-k}} + \Big( \frac{z}{w}\Big)^{k+1} \frac{k-\tfrac{1}{2}}{(z'+1)_k (w'+1)_{1-k}} \bigg)
		\\
		&=
		\frac{w}{u}\sum_{k = 1}^\infty \bigg( \Big( \frac{z}{w}\Big)^k \Big( \frac{1}{(z'+1)_{k-1} (w'+1)_{1-k}} - \frac{z}{2(z'+1)_k (w'+1)_{1-k}}\Big) \\
		& 
		\qquad + \Big( \frac{z}{w}\Big)^{k+1}\Big( -\frac{1}{(z'+1)_k (w'+1)_{-k}} +\frac{w}{2(z'+1)_k (w'+1)_{1-k}}\Big) \bigg)
		\\
		&=
		\frac{w}{u}\sum_{k = 1}^\infty \bigg( \Big( \frac{z}{w}\Big)^k \frac{1}{(z'+1)_{k-1} (w'+1)_{1-k}} - \Big(\frac{z}{w}\Big)^{k+1} \frac{1}{(z'+1)_k (w'+1)_{-k}} \bigg)
		\\
		&= \frac{z}{u} .
	\end{split}
	\end{equation*}
	So we obtain
	\begin{equation*}
		C_+ = \frac{z}{u (1 + z/w)} = \frac{z}{u} \sum_{k=0}^\infty \Big(-\frac{z}{w}\Big)^k .
	\end{equation*}
	Similar computations show that
	\begin{equation*}
		C_- = -\frac{w}{u (1 + w/z)} = -\frac{w}{u} \sum_{l=-\infty}^0 \Big(-\frac{z}{w}\Big)^l = \frac{z}{u}\sum_{k=-\infty}^{-1} \Big( - \frac{z}{w} \Big)^k . \qedhere
	\end{equation*}
\end{proof}

We are finally ready to give a proof of the commutation relation.

\begin{proof}[Proof of \cref{prop:commutation}]
	By \cref{thm:Bdressing}, $W^{-1}\ms{B}W - \widetilde{\ms{B}} = 0 \in \mf{h}$. Clearly, as $\widehat{\mf{h}}$ is a central extension of $ \mf{h}$, we have that $ W^{-1}\ms{B}W - \widetilde{\ms{B}} $ is central in $\widehat{\mf{h}}$ and hence irrelevant for any commutator. Therefore,
	\[
		\bigl[ \widetilde{\ms{B}}(z),\widetilde{\ms{B}}(w) \bigr]
		=
		W^{-1} \bigl[ \widetilde{\ms{B}}(z),\widetilde{\ms{B}}(w) \bigr] W
		=
		\bigl[ \widetilde{\ms{B}}(z),\widetilde{\ms{B}}(w) \bigr] ,
	\]
	where the last equality follows from the centrality of the commutator. The computation holds in $ \widehat{\mf{h}}^B$ as well, because $\widehat{\mf{h}}^B$ is embedded into $ \widehat{\mf{h}}$ via $ \iota$. The commutator of two $\mc{B}$ operators is now simply obtained multiplying by $u^2$.
\end{proof}

\section{Spin Gromov--Witten invariants in operator formalism}
\label{sec:sGW:operator:form}

In this section, we will return to the localisation formula, \cref{thm:GW:localised:to:H}, and insert the obtained expression for the double Hodge integrals in terms of $\mc{B}$-operators, \cref{thm:dH:as:VEV}, in order to obtain an vacuum expectation formula for the spin GW invariants. Such a formula gives a proof of the spin GW/H correspondence for the Riemann sphere, and we will discuss the spin GW/H correspondence for a general target curve assuming a degeneration formula. The operator formalism allows for explicit computations of spin GW invariants, that recover some known formulae and more.

\subsection{The main formula}
\label{subsec:sGW:operator:formula}

\begin{proposition}\label{prop:GW:as:quadratic:vev}
	The degree $d$, $(n+m)$-point function for equivariant spin GW invariants of $(\P^1,\mc{O}(-1))$ can be expressed in Fock space as a quadratic vacuum expectation value:
	\begin{equation}
		\sGWdisc_{d}(z,w; u)
		=
		\sum_{\mu \in \OP_d} \frac{2^{\ell(\mu)}}{\mf{z}_\mu} u^{-d} 
		\Braket{
			\prod_{i=1}^n \frac{\mc{B}(tz_i,uz_i)}{u}
			e^{\alpha^B_1} e^{\frac{u}{t} \frac{\mc{F}^B_3}{3}}
			\alpha^B_{-\mu}
		}
		\Braket{
			\prod_{j=1}^m \frac{\mc{B}(-tw_j, uw_j)}{u}
			e^{\alpha^B_1} e^{-\frac{u}{t} \frac{\mc{F}^B_3}{3}}
			\alpha^B_{-\mu}
		}
		.
	\end{equation}
\end{proposition}

\begin{proof}
	We start from the result of the localisation calculation, \cref{thm:GW:localised:to:H}:
	\[
			\sGWdisc_{d}(z,w;u)
			=
			\sum_{\mu \in \OP_d}
				\frac{2^{\ell (\mu)}}{\mf{z}_{\mu}} \frac{u^{-d}}{t^n (-t)^{m}} \\
				\times
				\prod_{k = 1}^{\ell(\mu)} \biggl( \frac{u}{t} \frac{(\frac{u\mu_k}{t})^{\mu_k'}}{\mu_k'!} \biggr)
				\dHdisc(\mu, tz; \tfrac{u}{t} ) 
				\prod_{k = 1}^{\ell(\mu)} \biggl( - \frac{u}{t} \frac{(- \frac{u\mu_k}{t})^{\mu_k'}}{\mu_k'!} \biggr)
				\dHdisc(\mu, -tw; -\tfrac{u}{t} ) .
	\]
	Substituting the expression of $\dHdisc$ as a vacuum expectation value of $\mc{B}$-operators, i.e. \cref{thm:dH:as:VEV}, gives
	\begin{equation*}
		\begin{split}
			\sGWdisc_{d}(z,w; u)
			=
			\sum_{\mu \in \OP_d} \frac{2^{\ell(\mu)}}{\mf{z}_\mu} u^{-d}
			&
			\prod_{k = 1}^{\ell (\mu)} \frac{(\frac{u\mu_k}{t})^{\mu_k'}}{\mu_k'!}
			\Braket{
				\prod_{i=1}^n \frac{\mc{B}(tz_i,uz_i)}{u}
				\prod_{k=1}^{\ell(\mu)} \mc{B}(\mu_k,\tfrac{u\mu_k}{t})
			}
			\\
			\times &
			\prod_{k = 1}^{\ell (\mu)} \frac{(- \frac{u\mu_k}{t})^{\mu_k'}}{\mu_k'!}
			\Braket{
				\prod_{j=1}^m \frac{\mc{B}(-tw_j,uw_j)}{u}
				\prod_{k=1}^{\ell(\mu)} \mc{B}(\mu_k, -\tfrac{u\mu_k}{t})
			} .
		\end{split}
	\end{equation*}
	Now we use \cref{prop:B:operators:odd:integers} for the $\mc{B}$-operators with odd integral parameter, that is
	\begin{equation*}
		\mc{B}(\mu, u\mu)
		=
		\frac{\mu'!}{(u \mu)^{\mu'}}
		e^{\alpha^B_1} e^{u\frac{\mc{F}^B_{3}}{3}} \alpha^B_{-\mu} e^{-u\frac{\mc{F}^B_{3}}{3}} e^{-\alpha^B_1} ,
	\end{equation*}
	to obtain the result.
\end{proof}

We can reduce the product of two correlators in \cref{prop:GW:as:quadratic:vev} to a single correlator using the following decomposition of the identity.

\begin{lemma} \label{lem:vev:identity}
	The identity operator on the Fock space of type B can be written as follows:
	\begin{equation}
		\Id
		=
		\sum_{d \geq 0} \sum_{\mu \in \OP_{d}} \frac{2^{\ell(\mu)}}{\mf{z}_{\mu}}
			\Ket{\alpha^{B}_{-\mu}} \Bra{\alpha^{B}_{\mu}} .
	\end{equation}
	More specifically,
	\begin{equation}
		\mc{P}_d = \sum_{\mu \in \OP_{d}} \frac{2^{\ell(\mu)}}{\mf{z}_{\mu}} \Ket{\alpha^{B}_{-\mu}} \Bra{\alpha^{B}_{\mu}}
	\end{equation}
	is the projection on the energy $d$ subspace.
\end{lemma}

\begin{proof}
	The proof follows from orthogonality of characters of the Sergeev group\footnote{
	Recall that the irreducible characters of the Sergeev group satisfy $\braket{\zeta^{\lambda_1},\zeta^{\lambda_2}} = 2^{(p(\lambda_1) + p(\lambda_2))/2} \delta_{\lambda_1,\lambda_2}$, where the product is defined as
	\begin{equation*}
		\braket{f, g} \coloneqq \sum_{\mu \in \OP_d} \frac{2^{-\ell(\mu)}}{\mf{z}_{\mu}} f_{\mu} g_{\mu}
	\end{equation*}
	for all $f,g$ in the spin class algebra $\mc{Z}_d$, that is the even part of the centre of the Sergeev algebra. See \cite{GKL21} and references therein.
	}:
	\begin{align*}
		\Id
		& =
			\sum_{d \geq 0} \sum_{\lambda \in \SP_{d}} \ket{\lambda} \bra{\lambda} \\
		& =
			\sum_{d \geq 0} \sum_{\lambda_1, \lambda_2 \in \SP_{d}} 
		\Biggl(
			\sum_{\mu \in \OP_d} \frac{2^{-\ell(\mu)} 2^{-\frac{p(\lambda_1) + p(\lambda_2)}{2}}}{\mf{z}_{\mu}}
			\zeta_{\mu}^{\lambda_1} \zeta_{\mu}^{\lambda_2}
		\Biggr) 
		\ket{\lambda_1} \bra{\lambda_2}
		\\
		& =
		\sum_{d \geq 0} \sum_{\mu \in \OP_d} \frac{2^{\ell(\mu)}}{\mf{z}_{\mu}}
		\Biggl(
			\sum_{\lambda_1 \in \SP_{d}} 
				2^{-\ell(\mu)} 2^{-\frac{p(\lambda_1)}{2}} \zeta_{\mu}^{\lambda_1} 
				\ket{\lambda_1} 
		\Biggr)
		\Biggl(
			\sum_{\lambda_2 \in \SP_{d}} 
			2^{-\ell(\mu)} 2^{-\frac{p(\lambda_2)}{2}} \zeta_{\mu}^{\lambda_2} 
			\bra{\lambda_2}
		\Biggr) ,
	\end{align*}
	which is equivalent to the statement by \cref{prop:boson:CnJ:action}.
\end{proof}

Combining \cref{prop:GW:as:quadratic:vev,lem:vev:identity}, we find that
\begin{equation}
	\begin{split}
		\sGWdisc_{d}(z,w; u)
		& =
		u^{-d}
		\Braket{
			\prod_{i=1}^n \frac{\mc{B}(tz_i,uz_i)}{u}
			e^{\alpha^B_1} e^{\frac{u}{t} \frac{\mc{F}_3^B}{3}} \mc{P}_d e^{-\frac{u}{t} \frac{\mc{F}_3^B}{3}} e^{\alpha^B_{-1}}
			\prod_{j=1}^m \frac{\mc{B}(-tw_j,uw_j)^*}{u}
		}
		\\
		& =
		u^{-d}
		\Braket{
			\prod_{i=1}^n \frac{\mc{B}(tz_i,uz_i)}{u}
			e^{\alpha^B_1} \mc{P}_d e^{\alpha^B_{-1}}
			\prod_{j=1}^m \frac{\mc{B}(-tw_j,uw_j)^*}{u}
		} ,
	\end{split}
\end{equation}
where in the second line we used that $\mc{F}^B_3$ commutes with $\mc{P}_d$. Let us define
\begin{equation}
	\sGWdisc(z,w;u,q) = \sum_{d=0}^\infty \sGWdisc_d (z,w;u) q^d ,
\end{equation}
and consider the operators $\ms{B}$ defined in the previous section (together with their adjoint)
\begin{align}
	\ms{B}(z) 
	&\coloneqq
	\frac{\mc{B}(tz,uz)}{u} ,
 &
	\ms{B}(z)^{\star} 
	&\coloneqq
	\ms{B}(z)^* \big|_{t \mapsto -t} ,
	\\
	\ms{B}_k
	&\coloneqq [z^{k+1}] \ms{B}(z) ,
	&
	\ms{B}^\star_k
	&\coloneqq [z^{k+1}] \ms{B}(z)^\star .
\end{align}
Then we obtain the final expression of for spin GW invariants as a unique vacuum expectation value.

\begin{theorem}\label{thm:sGW:as:vev}
	The $(n+m)$-point function for equivariant spin GW invariants of $(\P^1,\mc{O}(-1))$ can be expressed as a single vacuum expectation value:
	\begin{equation}
		\sGWdisc(z,w;u,q)
		=
		\Braket{
			\prod_{i=1}^n \ms{B}(z_i)
			e^{\alpha^B_1} \Bigl( \frac{q}{u} \Bigr)^H e^{\alpha^B_{-1}}
			\prod_{j=1}^m \ms{B}(w_j)^{\star}
		} .
	\end{equation}
	Here $H$ is the energy operator.
\end{theorem}

The proof of \cref{thm:sGW:as:vev:intro} now follows immediately.

\begin{proof}[Proof of \cref{thm:sGW:as:vev:intro}]
	Compare \cref{thm:sGW:as:vev} and \cref{def:nPlusmPointGW} to \cref{def:EquivTauFn}, the equivariant tau-func\-tion. The factors $-8$, $-4$, and $ 2$ in front of $ u$, $q$, and $ x$, $ x^\star$, respectively, in \cref{def:EquivTauFn} correspond to the factor $ (-1)^{g-1+d} 2^{3g-3+2d+n+m}$ in \cref{def:nPlusmPointGW}.
\end{proof}

As a consequence, we obtain an expression of the non-equivariant spin GW invariants as a vacuum expectation value.

\begin{proposition}\label{cor:GW:H:corr}
	The disconnected, degree $d$, $n$-point, stationary spin GW invariants of $(\P^1,\mc{O}(-1))$ with no degree zero components can be expressed in Fock space as:
	\begin{equation}
		\Braket{ \tau_{k_1} \cdots \tau_{k_n} }_{\emptyset,g,d}^{\bullet,\P^1,\mc{O}(-1)}
		=
		2^d
		\prod_{i=1}^n \frac{k_i!}{(-2)^{k_i}} [ z_i^{2k_i+1} ]
		\Braket{
			\frac{(\alpha_{1}^B)^d}{d!} \prod_{i=1}^n \hat{\mc{E}}_0^B(z_i) \frac{(\alpha_{-1}^B)^d}{d!}
		} .
	\end{equation}
	In particular, the right-hand side can be computed using the algorithm described in \cite[section 6.3]{GKL21}. 
\end{proposition}

\begin{proof}
	By definition of the $\mc{B}$-operators,
	\begin{equation*}
		\ms{B}(z)\big|_{(t,u)=(0,1)}
		=
		\sum_{k \in \Z} \sum_{j \ge k' + 1} \frac{\Gamma(\frac{1}{2}) \Gamma(2j)}{\Gamma(k+1)\Gamma(j + \frac{1}{2})}
			z^{j} [w^{2j-1}] \varsigma(w)^k \mc{E}^B_k(w) .
	\end{equation*}
	The $k$-sum can be restricted to $k \ge 0$ due to vanishing of $1/\Gamma(k+1)$, up to a certain amount of terms. In fact for some terms the Gamma function $\Gamma(2j)$ at the numerator simplifies the zero with a corresponding pole. We now show that these terms are finitely many, and that they do not contribute to the GW correlator. Let us analyse the case for $k$ negative and odd, the case for $k$ negative and even is similar. We have
	\begin{equation*}
		\frac{\Gamma(\frac{1}{2})\Gamma(2j)}{\Gamma(k+1)\Gamma(j + \frac{1}{2})}
		=
		\frac{\sqrt{\pi}}{\Gamma(j + \frac{1}{2})} \prod_{l=2}^{2j}(k + l) .
	\end{equation*}
	It is immediate to see that the other Gamma function at the denominator cannot contribute a zero (i.e., in the case $k$ negative and even, the argument is a half-integer as well). Notice that for fixed $k$, at least one of the factors $(k+l)$ vanishes for $j$ big enough, hence only finitely many $j$-terms can contribute. Let us now show that they do not contribute to the GW correlator. The term above vanishes as $l = -k > 0$, that is, whenever $j \geq 0$. Hence the surviving terms have $j < 0$, where $j$ is the resulting power of $z_i$ for that term. Therefore, these terms are never selected when taking a GW correlator, which corresponds to the coefficient of strictly positive powers for all the $z_i$.

	Now, by swapping the sums one obtains
	\begin{equation*}
		\ms{B}(z)\big|_{(t,u)=(0,1)}
		=
		\sum_{j \ge 1} \frac{\Gamma(\frac{1}{2}) \Gamma(2j)}{\Gamma(j + \frac{1}{2})}
			z^{j} [w^{2j-1}] \sum_{k \geq 0} \frac{\varsigma(w)^k}{k!} \mc{E}^B_k(w) ,
	\end{equation*}
	where the condition $k \leq 2j - 1$ can be disregarded as $\varsigma(w) = w + O(w^3)$. On the other hand, a simple application of the commutation relations given in \cref{prop:OP:VEV} shows that
	\begin{equation*}
		e^{\alpha^B_1} \mc{E}^B_0(w) e^{-\alpha^B_1}
		=
		\sum_{k \ge 0} \frac{\varsigma(w)^k}{k!} \mc{E}^B_k(w) .
	\end{equation*}
	Hence, we obtain the following expression for the non-equivariant limit of the $\ms{B}$-operators:
	\begin{equation*}
		\ms{B}(z)\big|_{(t,u)=(0,1)}
		=
		\sum_{j \ge 1} \frac{\Gamma(\frac{1}{2}) \Gamma(2j)}{\Gamma(j + \frac{1}{2})}
			z^{j} [w^{2j-1}] e^{\alpha^B_1} \mc{E}^B_0(w) e^{-\alpha^B_1} .
	\end{equation*}
	Moreover, we can use the Legendre duplication formula $2^{2x-1} \Gamma(x) = \sqrt{\pi} \frac{\Gamma(2x)}{\Gamma(x + \frac{1}{2})}$ to finally obtain
	\begin{equation*}
		[z^{k+1}] \ms{B}(z)\big|_{(t,u)=(0,1)}
		=
		2^{2k+1} k! [z^{2k+1}] e^{\alpha^B_1} \mc{E}^B_0(z) e^{-\alpha^B_1} .
	\end{equation*}
	Thus, \cref{thm:sGW:as:vev} implies that
	\begin{equation*}
		(-1)^{g-1+d} 2^{3g-3+n+2d}
		\Braket{ \tau_{k_1} \cdots \tau_{k_n} }_{\emptyset,g,d}^{\bullet,\P^1,\mc{O}(-1)}
		=
		\frac{1}{(d!)^2}
		\prod_{i=1}^n 2^{2k_i + 1} k_i! [ z_i^{2k_i+1} ]
		\Braket{
			(\alpha_{1}^B)^d \prod_{i=1}^n \hat{\mc{E}}_0^B(z_i) (\alpha_{-1}^B)^d
		} ,
	\end{equation*}
	where we changed from $ \mc{E}_0$ to $ \hat{\mc{E}}_0$ in order to exclude any degree zero components. A simplification of the prefactors gives the statement, together with the fact that $\sum_i k_i = g - 1 + d$.
\end{proof}

We can now compare this to formulae for Hurwitz numbers to prove \cref{thm:spin:GW:H:P1:intro}.

\begin{proof}[Proof of \cref{thm:spin:GW:H:P1:intro}]
	Compare \cref{prop:mixed:HN:VEV} to \cref{cor:GW:H:corr}, and convert to connected invariants. For this conversion, use that degree zero Hurwitz numbers are trivial: there is exactly one ramified (spin) cover of $ \P^1$ of degree $0$, and it has a disconnected (because empty) source.
\end{proof}

\subsection{Towards the full spin GW/H correspondence}
\label{subsec:full:spin:GW:H}

Here we prove \cref{thm:full:spin:GW:H}: the degeneration formula for spin GW invariants and the spin GW/H correspondence for the Riemann sphere imply the spin GW/H correspondence in full generality. The proof is adapted from \cite{OP06a} for the non-spin case.

Denote $\overline{c}_{\mu} \coloneqq \prod_{i=1}^{\ell(\mu)} \overline{c}_{\mu_i}$ where the $c_{\mu_i}$ are the spin completed cycles as introduced in \cref{eqn:completed:cycle}. The vectors $\overline{c}_{\mu}$ for all odd partitions $\mu$ form a basis of $\C\{ \OP \}$. Besides, using \cref{conj:degen:sGW}, denote:
\begin{equation}
	\widetilde{c}_{k} \coloneqq \sum_{\mu \in \OP} \widetilde{\kappa}_{k,\mu} \cdot \mu \in \C \{ \OP \}
	\qquad\quad \text{and} \qquad\quad
	\widetilde{c}_{\mu} \coloneqq \prod_{i=1}^{\ell(\mu)} \widetilde{c}_{\mu_i} .
\end{equation}
The first property of the numbers $\widetilde{\kappa}_{k,\mu}$ of \cref{conj:degen:sGW} implies that the elements $\widetilde{c}_{\mu}$ also form a basis. Moreover, $\widetilde{c}_{\mu} - \overline{c}_{\mu}$ is a linear combination of partitions of size smaller than $|\mu|$. Thus the transition matrix between these two basis is unitriangular. Now, for all $\xi \in \C$, we define the following quadratic form on the space $\Gamma$ of supersymmetric functions:
\begin{equation}
	q_\xi(\phi,\psi) \coloneqq \sum_{\lambda\in \SP } \xi^{|\lambda|} 2^{-p(\lambda)}{\dim}(V^\lambda) \phi({\lambda}) \psi({\lambda}) .
\end{equation}
The sum is convergent, as it is finite if $\phi$ and $\psi$ have bounded degrees, and for all $\xi \in \R_{+}$ the quadratic form $q_{\xi}$ is definite positive. Moreover, the second conjectural property of the numbers $\widetilde{\kappa}_{k,\mu}$ implies that 
\begin{equation}
	q_\xi(\widetilde{c}_{\mu},\widetilde{c}_{\nu}) = q_\xi(\overline{c}_{\mu},\overline{c}_{\nu}) ,
\end{equation}
for all odd partitions $\mu$ and $\nu$ (here we have implicitly used the identification of $\Gamma$ and $\C\{ \OP \}$). Therefore the transition matrix between the $\widetilde{c}_{\nu}$ and the $\overline{c}_{\mu}$ is unitriangular and orthonormal (for any positive $\xi \in \R_{+}$), thus it is the identity.

\subsection{Some explicit formulae}
\label{subsec:expl:formulae}

As an application of \cref{cor:GW:H:corr}, we give a new formula to compute correlators for $n = 1$ and arbitrary degree, and closed formulae for degree $d = 1,2$ and arbitrary $n$ recovering the result of Kiem--Li and Lee \cite{KL11a,KL11b,Lee13} originally conjectured by Maulik--Pandharipande \cite[section~2.4]{MP08} (see also \cite{KT17} for an approach using stable pairs). With similar computations, we prove a new formula in degree $d = 3$ and arbitrary $n$.

\begin{corollary}[{$1$-point spin GW}] \label{cor:sGW:n=1}
	Disconnected, $1$-point, stationary spin GW invariants of $(\P^1,\mc{O}(-1))$ with no degree zero components can be computed by
	\begin{equation}
		\Braket{ \tau_{k} }_{\emptyset,g,d}^{\bullet,\P^1,\mc{O}(-1)}
		=
		\frac{2^d k!}{( d! )^2} (-2)^{-k} [ z^{2k+1} ]
		\mc{U}_d(z) ,
		\qquad \qquad \text{for} \;\; d \geq 1 ,
	\end{equation}
	where $\mc{U}_d(z)$ is the first element of the vector computed as
	\begin{equation}\label{eqn:sGW:n=1}
			\begin{pmatrix}
			\mc{U}_d(z) \\
			\mc{V}_d(z)
		\end{pmatrix}
		=
		\left(
			\prod_{k=0}^{d-2} A_{d-k}(z)
		\right)
		\begin{pmatrix}
			\varsigma(z)\qoppa(z) \\
			\qoppa(z)
		\end{pmatrix}
		+
		\sum_{m=0}^{d-2} \left(
			\prod_{k=0}^{m-1} A_{d-k}(z)
		\right)
		t_{d-m}(z)
	\end{equation}
	with 
	\begin{equation}
		A_{p}(z)
		=
		\begin{pmatrix}
			\varsigma(z)^2 + p & (p-1)\varsigma(z)
			\\
			\varsigma(z) & p-1
		\end{pmatrix},
		\qquad 
		t_{p}(z)
		=
		\begin{pmatrix}
			(p-1) \varsigma(z)\qoppa(z)
			\\
			(p-1) \qoppa(z)
		\end{pmatrix}.
	\end{equation}
The product of matrices is to be taken with the first-indexed matrix on the left.
\end{corollary}

\begin{proof}
	The functions $\mc{U}_d(z), \mc{V}_d(z)$, which are odd and even functions respectively, are the vacuum expectations
	\begin{equation}
		\mc{U}_d(z) = \Braket{(\alpha_{1}^B)^d \hat{\mc{E}}_0^B(z) (\alpha_{-1}^B)^d} , 
		\qquad
		\mc{V}_d(z) = \Braket{(\alpha_{1}^B)^d \vphantom{\hat{\mc{E}}_1^B(z)} \mc{E}_1^B(z) (\alpha_{-1}^B)^{d-1}} .
	\end{equation}
	The proof is a straightforward computation. \Cref{eqn:sGW:n=1} is achieved by commuting once the right-most positive energy operator with the operator immediately to the right of it, using the commutation relations \cref{prop:OP:VEV}. The initial data are computed in the same way, and evaluating $\braket{\mc{E}_0^{B}(z)} = \qoppa(z)/\varsigma(z)$. See \cref{table:Ud:1pnt} for the first cases.
\end{proof}

\begin{corollary}[Low degree spin GW]
	Disconnected, $n$-point, stationary spin GW invariants of $(\P^1,\mc{O}(-1))$ with no degree zero components in low degree are given as follows.
	\begin{itemize}
		\item
		\textup{\textsc{Degree $1$:}}
		\begin{equation}
			\Braket{\tau_{k_1} \cdots \tau_{k_n}}^{\bullet, \P^1, \mc{O}(-1)}_{\emptyset,g,1}
			=
			\prod_{i=1}^{n}\frac{k_i!}{(2k_i+1)!}(-2)^{-k_i} .
		\end{equation}

		\item
		\textup{\textsc{Degree $2$:}}
		\begin{equation}
			\Braket{\tau_{k_1} \cdots \tau_{k_n}}^{\bullet, \P^1, \mc{O}(-1)}_{\emptyset,g,2}
			=
			\frac{1}{2}\prod_{i=1}^{n}\frac{2 \cdot k_i!}{(2k_i+1)!}(-2)^{k_i} .
		\end{equation}

		\item
		\textup{\textsc{Degree $3$:}}
		\begin{equation}
		\begin{split}
			\Braket{\tau_{k_1} \cdots \tau_{k_n}}^{\bullet, \P^1, \mc{O}(-1)}_{\emptyset,g,3}
			& =
			\frac{1}{9} \prod_{i=1}^n \frac{3 \cdot k_i!}{(2k_i+1)!} \Big( -\frac{9}{2} \Big)^{k_i} \\
			& \quad
			+ \frac{1}{18} \prod_{i=1}^n \frac{k_i!}{(2k_i+1)!} \big( (-2)^{-k_i} + 2 (-2)^{k_i} \big) .
		\end{split}
		\end{equation}
	\end{itemize}
\end{corollary}

We omit the proof of the above results\footnote{It can be downloaded from the preprints section of the third author's website.}, which again follows by commuting the operators in the vacuum expectation and applying \cref{prop:OP:VEV} for the commutation relations and the evaluation $\braket{\mc{E}_0^{B}(z)} = \qoppa(z)/\varsigma(z)$. We find that all generating series of spin GW invariants can be expressed as sums of products of hyperbolic functions. It is interesting to compare this phenomenon with the non-spin case: while the latter only involves several factors of the hyperbolic sine, the spin case involves a certain combinatorics of products of both the hyperbolic sine and the hyperbolic cosine. Because of well-known hyperbolic trigonometric identities, this combinatoric simplifies considerably.

\begin{remark}\label{rem:degzero}
	The difference between spin GW correlators with and without degree zero contributions can be explicitly computed via generating functions as follows:
	\begin{equation}
		\Braket{(\alpha_{1}^B)^d \prod_{i=1}^n \mc{E}_0^B(z_i) (\alpha_{-1}^B)^d}
		-
		\Braket{(\alpha_{1}^B)^d \prod_{i=1}^n \hat{\mc{E}}_0^B(z_i) (\alpha_{-1}^B)^d}
		=
		d! \cdot \mc{U}_0(z_1,\dots,z_n) .
	\end{equation}
\end{remark}

\section{The 2-BKP hierarchy}
\label{sec:2BKP}

In this last section, we discuss the integrability property of the equivariant tau-function, as well as the constraints given by the divisor and string equations. 

\subsection{Divisor and string equations}
\label{subsec:div:string}

\begin{proposition}[{Divisor equation}]
	We have
	\begin{equation}\label{eqn:divisor}
		[z_0] \sGWdisc_{g,d}(z_0, z_1, \dots, z_n,w_1, \dots, w_m ) 
		=
		2 \left( \frac{1}{24} + d + t \sum_{i=1}^n z_i \right) \sGWdisc_{g,d}(z_1, \dots, z_n,w_1, \dots, w_m ) .
	\end{equation}
\end{proposition}

\begin{proof}
	The proof is similar to the one of \cite[proposition 12]{OP06b}. Employing the neutral fermion description one obtains
	\begin{equation*}
		[z_0] \sGWdisc_{g,d}(z_0, z_1, \dots, z_n,w_1, \dots, w_m )
		=
		[z_0] \Braket{
			\ms{B}(z_0) \prod_{i=1}^n \ms{B}(z_i)
			e^{\alpha_{1}^{B}} \mc{P}_d e^{\alpha_{-1}^{B}}
			\prod_{j=1}^m \ms{B}(w_j)^{\star}
		} .
	\end{equation*}
	From \cref{def:B-ops} it follows that $[z] \ms{B}(z) = 2(\alpha^B_1 + \frac{1}{24} + \cdots )$, where the dots represent summands which are killed by the covacuum. Hence:
	\begin{equation*}
		[z] \Bra{\ms{B}(z)}
		=
		2 \Bra{\left( \alpha^B_1 + \frac{1}{24} \right)} = 2 \Bra{ \left( \alpha^B_1 + \frac{1}{24} + H \right) } ,
	\end{equation*} 
	where the energy operator $H$ can be added as it also gets annihilated by the covacuum. Notice that the following commutation holds: $[ \alpha^B_1 + H, \ms{B}(z) ] = tz \ms{B}(z)$. This allows us to move $\alpha^B_1 + H$ to the right of the $\ms{B}$-operators in the vacuum expectation value, generating a factor of $t \sum_i z_i$ in the process. Moreover, by employing the commutation relations $ [H, \alpha^B_1] = -\alpha^B_1 $ and $H \mc{P}_d = d \mc{P}_d$, one obtains that 
	\begin{equation*}
		(\alpha^B_1 + H )e^{\alpha^B_1} \mc{P}_d = d e^{\alpha^B_1} \mc{P}_d .
	\end{equation*}
	This generates the last factor of $d$ in \cref{eqn:divisor}.
\end{proof}

\begin{proposition}[{String equation}]
	We have
	\begin{multline}
		(-1)^{g-1+d} 2^{3g-3+n+m+2d}
		\Braket{
			e^{\tau_0(1)} \prod_{i=1}^n \tau_{k_i}(\mathbf{0}) \prod_{j=1}^m \tau_{l_j} (\boldsymbol{\infty})}_{g,d}^{\bullet, \P^1,\mc{O}(-1)
		} = \\
		=
		\Biggl[ \prod_{i=1}^n z_i^{k_i + 1} \prod_{j=1}^m w_j^{l_j + 1} \Biggr]
		e^{\sum_i z_i + \sum_j w_j}\sGWdisc_{g,d}(z_1, \dots, z_n,w_1, \dots, w_m) .
	\end{multline}
\end{proposition}

\begin{proof}
	The proof is the same as in \cite[proposition~13]{OP06b}, we repeat here some steps for reader's convenience. Recall that in the localised equivariant cohomology of $\P^1$ the identity class satisfies the relation $1 = \frac{\mathbf{0} - \boldsymbol{\infty} }{t}$, which can be used as 
	\begin{equation*}
		\ev_0^{*}(1) = \frac{\ev_0^{*}(\mathbf{0}) - \ev_0^{*}(\boldsymbol{\infty}) }{t} .
	\end{equation*}
	By the divisor equation above we compute:
	\begin{equation*}
	\begin{split}
		&
		(-1)^{g-1+d} 2^{3g-3+n+m+2d}
		\Braket{
			\tau_0(1)
			\prod_{i=1}^n \tau_{k_i}(\mathbf{0}) 
			\prod_{j=1}^m \tau_{l_j} (\boldsymbol{\infty})
		}_{g,d}^{\bullet, \P^1,\mc{O}(-1)}
		= \\
		&\quad =
		(-1)^{g-1+d} 2^{3g-3+n+m+2d}
		\Braket{
			\frac{\tau_0(\mathbf{0}) - \tau_0(\boldsymbol{\infty})}{t}
			\prod_{i=1}^n \tau_{k_i}(\mathbf{0})
			\prod_{j=1}^m \tau_{l_j} (\boldsymbol{\infty})
		}_{g,d}^{\bullet, \P^1,\mc{O}(-1)}
		\\
		&\quad =
		\frac{1}{2t}
		\Biggl[ z_0 \prod_{i=1}^n z_i^{k_i + 1} \prod_{j=1}^m w_j^{l_j + 1} \Biggr] \sGWdisc_{g,d}(z_0,z,w)
		-
		\frac{1}{2t}
		\Biggl[ w_0 \prod_{i=1}^n z_i^{k_i + 1} \prod_{j=1}^m w_j^{l_j + 1} \Biggr] \sGWdisc_{g,d}(z,w_0,w)
		\\
		&\quad =
		\frac{1}{2t}
		\Biggl[ \prod_{i=1}^n z_i^{k_i + 1} \prod_{j=1}^m w_j^{l_j + 1} \Biggr]
		\Bigg( 
			2 \Biggl( \frac{1}{24} + d + t \sum_{i=1}^n z_i \Biggr)
			-
			2 \Biggl( \frac{1}{24} + d - t \sum_{j=1}^m w_j \Biggr) 
		\Bigg)
		\sGWdisc_{g,d}(z,w)
		\\
		&\quad =
		\Biggl[ \prod_{i=1}^n z_i^{k_i + 1} \prod_{j=1}^m w_j^{l_j + 1} \Biggr]
			\Bigg( \sum_{i=1}^n z_i + \sum_{j=1}^m w_j \Bigg) 
			\sGWdisc_{g,d}(z,w) .
	\end{split}
	\end{equation*}
	Iterating this concludes the proof.
\end{proof}

\subsection{The hierarchy}
\label{subsec:hierachy}

In this section, we prove \cref{thm:2BKP}, i.e. that $ \tau $ is a tau-function of the $2$-BKP hierarchy.

\begin{proof}[Proof of \cref{thm:2BKP}]
	By \cref{thm:sGW:as:vev:intro},
	\begin{equation*}
		\tau(x,x^{\star};u,q)
		=
		\Braket{
			e^{\sum_i x_i \ms{B}_i } 
			e^{\alpha^B_1} \Bigl( \frac{q}{u} \Bigr)^H e^{\alpha^B_{-1}} 
			e^{\sum_j x_j^{\star} \ms{B}^\star_j }
		} .
	\end{equation*}
	By \cref{thm:Bdressing}, we find that $W^{-1} \ms{B}_k W = \tilde{\ms{B}}_k \coloneqq [z^{k+1}] \tilde{\ms{B}}(z)$ is a linear combination of $ \alpha_{2l+1}$ with $ l \geq k$. Therefore,
	\begin{equation*}
		W^{-1} e^{\sum_i x_i \ms{B}_i } W = \Gamma_+(t) ,
	\end{equation*}
	for a certain triangular linear transformation $ \{ x_i\} \mapsto \{ t_i\}$ induced by the above. Similarly,
	\begin{equation*}
		W^{\star} e^{\sum_i x^\star_i \ms{B}^\star_i } (W^\star)^{-1} = \Gamma_-(s) ,
	\end{equation*}
	where $ W^\star = W^*|_{t \mapsto -t}$ and $ \{ x^\star_i\} \mapsto \{s_i\}$ is obtained from $ \{ x_i\} \mapsto \{ t_i\}$ by inverting the equivariant parameter $ t \mapsto -t$. As moreover $ W$ is upper unitriangular by \cref{def:dressing}, $ \bra{W} = \bra{0} $ and $\ket{W^*} = \ket{0}$, so
	\begin{equation*}
		\begin{split}
			\tau(x,x^{\star};u,q)
			&=
			\Braket{
				W \Gamma_+ (t) W^{-1}
				e^{\alpha^B_1}  \Bigl( \frac{q}{u} \Bigr)^H e^{\alpha^B_{-1}} 
				(W^\star )^{-1} \Gamma_-(s) W^\star
			}
			\\
			&=
			\Braket{
				\Gamma_+ (t) W^{-1}
				e^{\alpha^B_1} \Bigl( \frac{q}{u} \Bigr)^H e^{\alpha^B_{-1}}
				(W^\star )^{-1} \Gamma_-(s)
			} ,
		\end{split}
	\end{equation*}
	which is of the shape of \cref{eq:2BKP:tau}, as $ \alpha_{\pm 1}^B, H \in \widehat{\mf{b}}_\infty$, and $ W $ is given by \cref{def:dressing}.
\end{proof}

\printbibliography

\end{document}